\documentclass[a4paper]{amsart}

\usepackage{amssymb}
\usepackage{amsthm}  
\usepackage{amsmath} 
\usepackage{amscd} 
\usepackage[all]{xypic}
\usepackage{url}
\usepackage{multirow}
\usepackage{hyperref}
\hypersetup{linktocpage}

\newcommand{\om}{\omega}

\newcommand{\ba}{\mathcal{K}} 
\newcommand{\ga}{\mathcal{G}}

\newcommand{\fg}{\frak g}

\newcommand{\fk}{\frak k}

\newcommand{\fl}{\frak l}
\newcommand{\fh}{\frak h}
\newcommand{\fn}{\frak n}
\newcommand{\fm}{\frak m}
\newcommand{\fs}{\frak s}

\newcommand{\Ad}{{\rm Ad}}

\newcommand{\Aut}{{\rm Aut}}
\newcommand{\id}{{\rm id}}
\newcommand{\conj}{{\rm conj}}

\newcommand{\ad}{{\rm ad}}

\newtheorem{thm}{Theorem}[section]
\newtheorem{prop}[thm]{Proposition}
\newtheorem{lem}[thm]{Lemma}
\newtheorem{cor}[thm]{Corollary}
\theoremstyle{definition}
\newtheorem{def*}{Definition}[]
\theoremstyle{remark}
\newtheorem{rem}[thm]{Remark}
\newtheorem*{exam}{Example}

\begin{document}
\title{Cartan geometries modeled on skeletons and morphisms induced by extension functors}
\author{Jan Gregorovi\v c}
\address{E.\v Cech  Institute,  Mathematical  Institute  of Charles  University,  Sokolovsk\'a 83, Praha 8 - Karl\'in, Czech Republic}
 \email{jan.gregorovic@seznam.cz}
 \thanks{The author was supported by the Grant agency of the Czech republic under the grant
P201/12/G028. I would like to thank J. Slov\'ak and L. Zalabov\' a for their suggestions and reading of the preliminary version of this article.}
\keywords{Categories of Cartan geometries; skeletons; extension functors; automorphisms; homothety; equivalence of Riemannian metrics}
\subjclass[2010]{53C05; 	53C30,53C29 }

\begin{abstract}
The extension functors between categories of Cartan geometries can be used to define different categories of Cartan geometries with additional morphisms. The Cartan geometries modeled on skeletons can be used for the description of such categories of Cartan geometries and therefore we develop the theory of Cartan geometries modeled on skeletons, in detail.  In particular, we show that there are functors from the categories of Cartan geometries with morphisms induced by extension functors to certain categories of Cartan geometries modeled on skeletons with the usual morphisms. We study how to compute (infinitesimal) automorphisms of Cartan geometries modeled on skeletons and how to compute the morphisms induced by the extension functors for particular Cartan geometries. Some examples and applications of the morphisms induced by extension functors are presented in the setting of the Riemannian geometries.
\end{abstract}
\maketitle
\tableofcontents

\section*{Introduction}

There are several possibilities, how to set up the theory of Cartan geometries. In this article, we will consider only finite dimensional Lie groups, algebras and their representations over real numbers and Cartan geometries on connected smooth manifolds. Then there the following algebraic objects that can be used as the models for the Cartan geometries.

In \cite{S}, the author models Cartan geometries on the infinitesimal Klein geometries $(\fg,\fh)$ assuming that there is triple $(\fg,H,\Ad)$, where $\fg$ is a Lie algebra with Lie subalgebra $\fh$, $H$ is a Lie group with Lie algebra $\fh$ and $\Ad$ is a representation of $H$ in the automorphism group $\Aut(\fg)$ of the Lie algebra $\fg$ extending the adjoint representation $\Ad_H$ of $H$ on $\fh$. In general, we will us the notation $\Ad_G$ for the adjoint representation of $G$ on its Lie algebra $\fg$ in this article.  

In \cite{parabook}, the authors model Cartan geometries on Klein geometries $(G,H)$ for a Lie group $G$ with closed Lie subgroup $H$. 

In \cite{Mor}, the author models Cartan geometries on the following objects, which we will need to use in this article.

\begin{def*}\label{defskel}
We say that $(\fk,L,\rho)$ is a \emph{skeleton} if

\begin{enumerate}
\item $L$ is a Lie group with Lie algebra $\fl$,
\item $\fk$ is a vector space with vector subspace $\fl$,
\item and $\rho$ is a representation of $L$ on $\fk$ such that $\rho|_\fl=\Ad_L$.
\end{enumerate}
\end{def*}

Since the triples $(\fg,H,\Ad)$ corresponding to the infinitesimal Klein geometries $(\fg,\fh)$ and the triples $(\fg,H,\Ad_G|_H)$ corresponding to the Klein geometries $(G,H)$ are skeletons, we can uniformly define the Cartan geometries in the following way.

\begin{def*}\label{defcar}
Let $(\fk,L,\rho)$ be a skeleton. We say that $(\ba\to M,\om)$ is a \emph{Cartan geometry of type $(\fk,L,\rho)$} if

\begin{enumerate}
\item $\ba\to M$ is a principal $L$--bundle over the smooth connected manifold $M$, where we denote by $r^l$ the right action of $l\in L$ on $\ba$ and by $\zeta_X$ the fundamental vector field of $X\in \fl$ on $T\ba$,
\item and $\om$ is a $\fk$--valued one form on $\ba$ such that $$(r^l)^*(\om)=\rho(l)^{-1}\circ \om$$ holds for all $l\in L$,  $\om(\zeta_X)=X$ holds for all $X\in \fl$ and $\om(u): T_u\ba\to \fk$ is a linear isomorphism for all $u\in \ba$.
\end{enumerate}

We say that $\om$ is a \emph{Cartan connection of type $(\fk,L,\rho)$}.
\end{def*}

In \cite{parabook} and \cite{S}, the categories of Cartan geometries are defined using the following morphisms.

\begin{def*}\label{defcar}
Let $(\ba\to M,\om)$ and $(\ba'\to M,'\om')$ be Cartan geometries of type $(\fk,L,\rho)$. We say that a $L$--bundle morphism $\Phi: \ba\to \ba'$ is a \emph{morphism of Cartan geometries of type $(\fk,L,\rho)$} if $$(\Phi)^*(\om')=\om.$$ 

We denote by $\mathcal{C}^{(\fk,L,\rho)}$ the category of Cartan geometries of type $(\fk,L,\rho)$ with these morphisms of Cartan geometries of type $(\fk,L,\rho)$.
\end{def*}

Clearly, the categories of Cartan geometries of type $(G,H)$ or $(\fg,\fh)$ (defined in \cite{parabook} and \cite{S}) are equivalent to the categories $\mathcal{C}^{(\fg,H,\Ad_G|_H)}$ or $\mathcal{C}^{(\fg,H,\Ad)}$, respectively. In the section \ref{sec1}, we review the basic properties of Cartan geometries of type $(\fk,L,\rho)$ and discuss which properties of the Cartan geometries modeled on (infinitesimal) Klein geometries are forgotten by considering them as Cartan geometries modeled on skeletons.

The crucial objects for the definition of different categories of Cartan geometries are extensions between skeletons.

\begin{def*}\label{defext}
We say that $(\alpha,i)$ is a \emph{morphism of the skeleton $(\fg,H,\sigma)$ to the skeleton $(\fk,L,\rho)$} if

\begin{enumerate}
\item $i$ is a Lie group homomorphism $H\to L$,
\item $\alpha$ is a linear map $\fg\to \fk$ extending $di: \fh \to \fl$ such that $$\alpha\circ \sigma(h)=\rho(i(h))\circ\alpha$$ holds for all $h\in H$.
\end{enumerate}

We say that a morphism $(\alpha,i)$ of $(\fg,H,\sigma)$ to $(\fk,L,\rho)$ is an \emph{extension of the $(\fg,H,\sigma)$ to $(\fk,L,\rho)$} if $\alpha$ provides isomorphism of $\fg/\fh$ and $\fk/\fl$.

We say that an extension $(\alpha,i)$ is \emph{injective, bijective}, if $i$ is injective, bijective, respectively.
\end{def*}

There is a natural source of extensions in the case of Klein geometries. If $G$ is a Lie subgroup of Lie group $K$ and $L$ is closed subgroup of $K$ such that $G$ acts transitively on $K/L$ and if we denote by $\iota$ the natural inclusion $G\to K$, then the pair $(d\iota,\iota|_{G\cap L})$ is extension of $(\fg,G\cap L,\Ad_G|_{G\cap L})$ to $(\fk,L,\Ad_{K}|_L)$.

The extensions between skeletons corresponding to Klein geometries provide extension functors (described in \cite[Section 1.5.15.]{parabook}) between the corresponding categories of Cartan geometries. We show (see Theorem \ref{cartoprinc}) that the extensions $(\alpha,i)$ of $(\fg,H,\sigma)$ to $(\fk,L,\rho)$ provide extension functors $$\mathcal{F}_{(\alpha,i)}: \mathcal{C}^{(\fg,H,\sigma)}\to \mathcal{C}^{(\fk,L,\rho)}.$$

There is the following formula for the composition of morphisms $\Phi,\Phi'$ with the extension functors $\mathcal{F}_{(\alpha,i)},\mathcal{F}_{(\alpha',i')}$ in/between the appropriate categories of Cartan geometries:
$$(\Phi \circ \mathcal{F}_{(\alpha,i)})\circ(\Phi'\circ \mathcal{F}_{(\alpha',i')})=(\Phi\circ \mathcal{F}_{(\alpha ,i)}(\Phi'))\circ (\mathcal{F}_{(\alpha\circ \alpha',i\circ i')}).$$
Therefore the following category of Cartan geometry with additional morphisms induced by the extension functors is well--defined.

\begin{def*}\label{symdef0}
The category $\mathcal{C}$ of Cartan geometries is a category of all Cartan geometries modeled on all skeletons with morphisms that are all possible compositions $\Phi \circ \mathcal{F}_{(\alpha,i)}$ of morphisms $\Phi$ of Cartan geometries modeled on skeletons with extension functors $\mathcal{F}_{(\alpha,i)}$.
\end{def*}

It is convenient to use the following equivalent (see Proposition \ref{altdef}) definition of (iso)morphisms in the category $\mathcal{C}$.

\begin{def*}\label{symdef2}
Let $(\alpha,i)$ be an extension of $(\fg,H,\sigma)$ to $(\fk,L,\rho)$, let $(\ga\to I,\gamma)$ be a Cartan geometry of type $(\fg,H,\sigma)$ and let $(\ba\to M,\om)$ be a Cartan geometry of type $(\fk,L,\rho)$. We say that a principal bundle morphism $\Phi:\ga\to \ba$ over $i: H\to L$ is a \emph{$(\alpha,i)$--morphism from $(\ga\to I,\gamma)$ to $(\ba\to M,\om)$} if $$(\Phi)^*(\om)=\alpha\circ \gamma.$$

We say that a $(\alpha,i)$--morphism $\Phi$ is a \emph{$(\alpha,i)$--isomorphism} if both $(\alpha,i)$ and $\Phi$ are bijective.
\end{def*}

In particular, we show (see Proposition \ref{altdef}) that the bijective extensions between different skeletons provide equivalences of the corresponding categories. This means that $(\alpha,i)$--isomorphisms for the bijective extensions of $(\fk,L,\rho)$ to $(\fk,L,\rho)$ are equivalences of Cartan geometries in the category $\mathcal{C}^{(\fk,L,\rho)}$. Therefore as we show on in the section \ref{sec6}, the $(\alpha,i)$--isomorphisms can be applied for resolving the problem of (global) equivalence of geometric structure related to Cartan geometries of type $(\fk,L,\rho)$. Clearly, the bijective extensions of $(\fk,L,\rho)$ to $(\fk,L,\rho)$ form a group and we will use the following notation.

\begin{def*}
We denote by $Ext(\fk,L,\rho)$ the group of bijective extensions of $(\fk,L,\rho)$ to $(\fk,L,\rho)$,
\end{def*}

The Proposition \ref{extcomp} shows, how to compute the group $Ext(\fk,L,\rho)$ as a subgroup of $Gl(\fk)\times \Aut(L)$ for a large class of skeletons.

There are many possibilities, how to extend the categories $\mathcal{C}^{(\fk,L,\rho)}$ inside the category $\mathcal{C}$ of Cartan geometries. These possibilities depend on the choice of the subgroup $\mathcal{S}$ of the group $Ext(\fk,L,\rho)$, because the particular choices of $\mathcal{S}$ can provide additional properties for the morphisms in these categories. For example, the choice $\mathcal{S}=\{(\alpha,\id)\in Ext(\fk,L,\rho)\}$ provides a category with all $(\alpha,\id)$--morphisms that are $L$--bundle morphisms.

\begin{def*}\label{symdef1}
Let $\mathcal{S}$ be a subgroup of the group $Ext(\fk,L,\rho)$. We say that $(\alpha,i)$--morphisms for $(\alpha,i)\in \mathcal{S}$ are \emph{$\mathcal{S}$--morphisms} and denote by $\mathcal{SC}^{(\fk,L,\rho)}$ the category of Cartan geometries of type $(\fk,L,\rho)$ with $\mathcal{S}$--morphisms.

We denote by $\mathcal{S}\Aut(\ba,\om)$ the group of all $(\alpha,i)$--automorphisms of $(\ba\to M,\om)$ for all $(\alpha,i)\in \mathcal{S}$ and we say that $\mathcal{S}\Aut(\ba,\om)$ is an \emph{$\mathcal{S}$--automorphisms group of $(\ba\to M,\om)$}.
\end{def*}

Let us point out that $\{(\id,\id)\}\mathcal{C}^{(\fk,L,\rho)}=\mathcal{C}^{(\fk,L,\rho)}$ holds for the trivial choice $\mathcal{S}=\{(\id,\id)\}$. In particular, the map $\mathcal{S}\Aut(\ba,\om)\to \mathcal{S}$ assigning to an $(\alpha,i)$--automorphism the extension $(\alpha,i)$ is a group homomorphism with kernel $\Aut(\ba,\om)$, i.e., the Lie group of automorphisms $\Aut(\ba,\om)$ of $(\ba\to M,\om)$ is a normal subgroup of the group $\mathcal{S}\Aut(\ba,\om)$.

For skeletons $(\fg,H,\Ad_G|_H)$ corresponding to Klein geometries $(G,H)$, there are several distinguished choices of the subgroup $\mathcal{S}$ of $Ext(\fg,H,\Ad_G|_H)$. For $(\alpha,i)\in Ext(\fg,H,\Ad_G|_H)$, the map $$R_\alpha(X,Y):=\ad_\fg(\alpha(X))\alpha(Y)-\alpha(\ad_\fg(X)Y)$$ measures for $X,Y\in \fg$ the failure of $\alpha$ to be a Lie algebra automorphism of $\fg$ and we can consider the extensions $(\alpha,i)$ with particular image of $R_\alpha$. For example, the subgroup $$\{(\alpha,i)\in Ext(\fg,H,\Ad_G|_H): Im(R_\alpha)\subset \fh\}$$ corresponds to the extension functors that preserve the torsion--freeness of Cartan geometries of type $(G,H)$. Further, there is the subgroup $$\{(\alpha,i)\in Ext(\fg,H,\Ad_G|_H): \alpha\in \Aut(\fg)\}$$ of extensions for which $Im(R_\alpha)=0$. We will see that the subgroup $\{(d\sigma,\sigma|_H): \sigma\in \Aut(G),\ \sigma(H)\subset H\}$ of $Ext(\fg,H,\Ad_G|_H)$ generates the additional $Ext(\fg,H,\Ad_G|_H)$--automorphisms of the flat model of Cartan geometries of type $(G,H)$, see Corollary \ref{modhom}.

In the section \ref{sec2}, we investigate the properties of the group $\mathcal{S}\Aut(\ba,\om)$ in detail. We show (see Theorem \ref{main1}) that if $\mathcal{S}$ is a Lie group with Lie algebra $\fs$, then there is a fully--faithful functor $\mathcal{F}_{\mathcal{S}}$ (see Definition \ref{defskelpr}) from the category $\mathcal{SC}^{(\fk,L,\rho)}$ to a particular category of Cartan geometries with usual morphisms such that $$\Aut(\mathcal{F}_{\mathcal{S}}(\ba\to M,\om))=\mathcal{F}_{\mathcal{S}}(\mathcal{S}\Aut(\ba,\om)).$$ Therefore $\mathcal{S}\Aut(\ba,\om)$ is a Lie group of dimension at most $dim(\Aut(\ba,\om))+dim(\frak{s})$. Let us point out that even for Cartan geometries modeled on Klein geometries, the functor $\mathcal{F}_{\mathcal{S}}$ ends in the categories of Cartan geometries modeled on skeletons, which is one of the main reasons, why we work with Cartan geometries modeled on skeletons in this article.

In the section \ref{sec5}, we discuss how to compute the automorphism and $\mathcal{S}$--automorphism groups for the Cartan geometries modeled on skeletons using the functor $\mathcal{F}_{\mathcal{S}}$, see Proposition \ref{autcomp0}, and Theorems \ref{autcomp} and \ref{homcomp}.

In the section \ref{sec6}, we show that if we view Riemannian geometries as Cartan geometries of type $(\mathbb{R}^n\oplus \frak{so}(n),O(n),\Ad_{\mathbb{R}^n\rtimes O(n)}|_{O(n)})$, then the homotheties of Riemannian geometries corresponds to $\mathcal{S}$--morphisms of for certain choice of $\mathcal{S}\subset Ext(\mathbb{R}^n\oplus \frak{so}(n),O(n),\Ad_{\mathbb{R}^n\rtimes O(n)}|_{O(n)})$, see Proposition \ref{riehom}. Further, we show that the problem to determine all (globally) non--equivalent of Riemannian metrics with the same Levi--Civita connection can be solved by the computation of certain $\mathcal{S}$--automorphism groups on the holonomy reduction, see Theorems \ref{autext} and \ref{meteq}.

\section{The basic properties of Cartan geometries modeled on skeletons}\label{sec1}

Let us list the main differences and similarities between the properties of Cartan geometries modeled on skeletons and the Cartan geometries modeled on (infinitesimal) Klein geometries. Let us recall that we will denote by $(\fg,H,\Ad_G|_H)$ the skeletons corresponding to Klein geometries $(G,H)$.

\begin{enumerate}
\item The map assigning to a (infinitesimal) Klein geometry the corresponding skeleton is neither injective nor surjective.

\begin{itemize}
\item There are skeletons $(\fk,L,\rho)$, which do not correspond to any (infinitesimal) Klein geometry. For example, consider a Klein geometry $(G,P)$, where $G$ is semisimple and $P$ parabolic subgroup of $G$, and denote by $\fg^i$ the corresponding $P$--invariant filtration of the Lie algebra $\fg$, c.f. \cite{parabook}. If $\fg\neq \fg^{-1}$, then $(\fg^{-1},P,\Ad_G|_P)$ is a skeleton corresponding to no Klein geometry, because any compatible Lie algebra structure on $\fg^{-1}$ is completely determined by the restricted root space decomposition of $\fg$ and thus there is none if $\fg\neq \fg^{-1}$.
\item There are skeletons $(\fk,L,\rho)$, which correspond to several Klein geometries. For example the Klein geometries $(O(n+1),O(n))$, $(O(n,1),O(n))$ and $(\mathbb{R}^n\rtimes O(n),O(n))$ share the same skeleton $(\mathbb{R}^n\oplus \frak{so}(n),O(n),\Ad_{\mathbb{R}^n\rtimes O(n)}|_{O(n)})$. In general, all mutants of infinitesimal Klein geometry $(\fg,\fh)$ (see \cite{S}) share the same skeleton.
\end{itemize}

\item It is clear from point (1) that the main information we are losing by considering the Cartan geometries modeled on skeletons are the notions of curvature and flat model of the Cartan geometry. In fact, these two notions are closely related. If we choose Cartan geometry $(\bar \ba\to \bar M,\bar \om)$ of type $(\fk,L,\rho)$ as our flat model, then we can define curvature of Cartan geometry $(\ba\to M,\om)$ of type $(\fk,L,\rho)$ as
$$\kappa(\om,\bar \om)(X,Y):=d\om(\om^{-1}(X),\om^{-1}(Y))-d\bar \om(\bar \om^{-1}(X),\bar \om^{-1}(Y)).$$

Indeed, the usual curvature of Cartan geometries of type $(G,H)$ is obtained by the choice of the Cartan geometry $(G\to G/H,\om_G)$ as the flat model, where $\om_G$ is the Maurer--Cartan form of $G$.

\item Let us now discuss the trivial morphisms in the category $\mathcal{C}^{(\fk,L,\rho)}$, i.e., the morphisms with the identity as the underlying map. For this, we need the notion of effective skeleton.

\begin{def*}\label{defskelpr}
We say that the maximal normal subgroup $N$ of $L$ such that $(id-\rho(n))\fk\subset \fn$ holds for all $n\in N$ is a \emph{kernel of $(\fk,L,\rho)$}.

We say that a skeleton $(\fk,L,\rho)$ is \emph{effective} if the kernel of $(\fk,L,\rho)$ is trivial.
\end{def*}

Let us prove that there are no trivial morphisms in the category $\mathcal{C}^{(\fk,L,\rho)}$ in the case of effective skeletons.

\begin{thm} \label{rigid}
Let $N$ be the kernel of $(\fk,L,\rho)$ and let $\phi_1$ and $\phi_2$ be two morphisms between two Cartan geometries $(\ba\to M,\om)$ and $(\ba'\to M,'\om')$ of type $(\fk,L,\rho)$ which cover the same base mapping $M\to M'$. Then there is smooth map $\psi: \ba\to N$ to the kernel $N$ of $(\fk,L,\rho)$ such that $\phi_2(u)=\phi_1(u)\psi(u)$ holds for all $u\in \ba$.

In particular, if $(\fk,L,\rho)$ is effective, then $\phi_2=\phi_1$.
\end{thm}
\begin{proof}
The Theorem can be proved by the proof of the analogous statement in the case of Cartan geometries modeled on Klein geometries, \cite[Proposition 1.5.3]{parabook}. Indeed, the proof requires only the skeleton $(\fg,H,\Ad_G|_H)$ instead of the Klein geometry $(G,H)$ to prove the fact that the map $\psi$ has values in group with the defining properties of the kernel $N$.
\end{proof}

\item The characterization of (infinitesimal) automorphisms remains the same. 

\begin{thm}\label{Liegp}
The automorphism group $\Aut(\ba,\om)$ of the Cartan geometry $(\ba\to M,\om)$ of type $(\fk,L,\rho)$ is a Lie group of dimension at most $dim(\fk)$. In particular, the Lie algebra of $\Aut(\ba,\om)$ consists of infinitesimal automorphisms, i.e., the vector fields $\xi$ on $T\ba$ such that $\mathcal{L}_\xi\om=0$ and $(r^l)^*(\xi)=\xi$ hold for all $l\in L$, which are complete.
\end{thm}
\begin{proof}
The Theorem can be proved by the proof of the analogous statement in the case of Cartan geometries modeled on Klein geometries, see \cite[Theorem 1.5.11]{parabook}.
\end{proof}

We will show in the section \ref{sec5}, how to compute the (infinitesimal) automorphisms of the Cartan geometries of type $(\fk,L,\rho)$. In fact, we will work in the category $\mathcal{SC}^{(\fk,L,\rho)}$ in the section \ref{sec5} and the results for the category $\mathcal{C}^{(\fk,L,\rho)}$ will follow for the choice $\mathcal{S}=\{(\id,\id)\}$.

\item There are several classes of morphisms of skeletons that can be used for unified description of natural connections associated to Cartan geometries modeled on skeletons. This description generalizes the results \cite[Theorem 1.5.6.]{parabook} and \cite[Lemma 1.5.15]{parabook} that hold for the Cartan geometries modeled on Klein geometries.

\begin{thm}\label{cartoprinc}
Let $(\alpha,i)$ be a morphisms of the skeleton $(\fg,H,\sigma)$ to the skeleton $(\fk,L,\rho)$ and let $(\pi: \ga\to M,\gamma)$ be a Cartan geometry of type $(\fg,H,\sigma)$. Then:

\begin{enumerate}
\item If $(\alpha,i)$ is extension, then there is unique Cartan connection $\om^\alpha$ of type $(\fk,L,\rho)$ on the principal $L$--bundle $\ga\times_{i(H)}L\to M$ satisfying $$\iota^*(\om^\alpha)=\alpha\circ \gamma,$$ where $\iota$ is the natural inclusion of $\ga\to \ga\times_{i(H)}L$. The assignment $$\mathcal{F}_{(\alpha,i)}(\ga\to M,\gamma):=(\ga\times_{i(H)}L\to M,\om^\alpha)$$
and the formula
$$\mathcal{F}_{(\alpha,i)}(\Phi)\circ r^l \circ \iota(u):=r^l\circ \iota(\Phi(u))$$
for morphism $\Phi\in \mathcal{C}^{(\fg,H,\sigma)}$, $l\in L$ and $u\in \ga$ together define an extension functor $\mathcal{F}_{(\alpha,i)}: \mathcal{C}^{(\fg,H,\sigma)}\to \mathcal{C}^{(\fk,L,\rho)}.$
\item If $\fk=\fl$, then there is unique principal connection $\om^\alpha$ on the principal $L$--bundle $\ga\times_{i(H)}L\to M$ such that $$\iota^*(\om^\alpha)=\alpha\circ \gamma$$ holds. The assignment $\gamma\mapsto \om^\alpha$ is a functor from the category $\mathcal{C}^{(\fg,H,\sigma)}$ to the category of principal bundle connections.
\item If $L=Gl(\fm)$ and $\fk=\fl=\frak{gl}(\fm)$ for some vector space $\fm$, then there is unique linear connection $\nabla^\alpha$ on the vector bundle $\ga\times_{i(H)}\fm \to M$ associated to the natural action of $L$ on $\fm$ such that $$(T\pi)^*(\nabla^\alpha s)(\xi)=D_{\gamma(\xi)}s+\alpha\circ \gamma(\xi)(s)$$ holds for all sections $s$ of $\ga\times_{i(H)}\fm$ and all $\xi\in T\ga$, where the fundamental derivative $D_{t}s(u)$ is given by the directional derivative of $s$ in direction $\gamma^{-1}(\gamma(\xi)(u))$ at $u\in \ga$. The assignment $\gamma\mapsto \nabla^\alpha$ is a functor from the category $\mathcal{C}^{(\fg,H,\sigma)}$ to the category of linear connections on the associated vector bundles.
\end{enumerate}

Moreover, if $\fk$ has a chosen Lie algebra structure, then the curvature of $\om^\alpha$ or $\nabla^\alpha$ as the function $\kappa_\alpha: \ga\to \wedge^2(\fg/\fh)^*\otimes \fk$ is given by formula $$\kappa_\alpha(X+\fh,Y+\fh)=\alpha(d\gamma(\gamma^{-1}(X),\gamma^{-1}(Y)))+\ad_\fk(\alpha(X))(\alpha(Y))$$ for $X,Y\in \fg/\fh$.
\end{thm}
\begin{proof}
The existence, uniqueness and functoriality of the construction of $\om^\alpha$ can be proved by the proofs of the analogous statements in the case of Cartan geometries modeled on Klein geometries, see \cite[Theorem 1.5.6]{parabook} and \cite[Lemma 1.5.15]{parabook}. Indeed, the requirements on the morphism $(\alpha,i)$ of skeletons are precisely those that are required by the \cite[Theorem 1.5.6]{parabook} and \cite[Lemma 1.5.15]{parabook}. Then we can follow the \cite[Section 1.5.8.]{parabook} to show that the fundamental derivative is well--defined and has the same properties as the fundamental derivative on Cartan geometries modeled on Klein geometries. Therefore the claim for $\nabla^\alpha$ then follows from the usual relation between the principal connection and the linnear connections on the associated vector bundles (cf.  \cite[Proposition 1.3.4]{parabook}).

The claim about $\kappa_\alpha$ follows by the same arguments as in proofs of \cite[Corollary 1.5.7]{parabook} and \cite[Proposition 1.5.16]{parabook} even in the cases, when the curvature of the Cartan geometry $(\ga\to M,\gamma)$ is not available.
\end{proof}

Let us remark that the tractor connections in \cite{parabook} for Cartan geometries of type $(\fg,H,\Ad_G|_H)$ and representations $\lambda: G\to Gl(\fm)$ correspond to morphisms $(d\lambda,\lambda|_H)$ of skeleton $(\fg,H,\Ad_G|_H)$ to skeleton $(\frak{gl}(\fm),Gl(\fm),\Ad_{Gl(\fm)})$.

\item As in the case of Cartan geometries modeled on Klein geometries c.f. \cite[Proposition 1.5.15.]{parabook}, the extension functors can be used for the complete classification of homogeneous Cartan geometries of type $(\fk,L,\rho)$. Indeed, in the same way as in \cite[Proposition 1.5.15.]{parabook}, it follows that, if $G$ is a transitive subgroup of the automorphisms group of $(\ba\to M,\om)$, then there is injective extension $(\alpha,i)$ of effective skeleton $(\fg,H,\Ad_G|_H)$ corresponding to Klein geometry $(G,H)$ to $(\fk,L,\rho)$ such that the Cartan geometry $(\ba\to M,\om)$ is isomorphic to $\mathcal{F}_{(\alpha,i)}(G\to G/H,\om_G)$. Conversely,  $\mathcal{F}_{(\alpha,i)}(G\to G/H,\om_G)$ is a homogeneous Cartan geometry of type $(\fk,L,\rho)$ for any extension $(\alpha,i)$ of skeleton $(\fg,H,\Ad_G|_H)$ corresponding to Klein geometry $(G,H)$ to $(\fk,L,\rho)$. Indeed, $G$ is a transitive subgroup of the automorphisms group of $\mathcal{F}_{(\alpha,i)}(G\to G/H,\om_G)$, because $\mathcal{F}_{(\alpha,i)}$ is a functor.

\item Finally, let us show, how to compute the group $Ext(\fk,L,\rho)$ for the effective skeletons $(\fk,L,\rho)$. 

\begin{prop}\label{extcomp}
If $(\fk,L,\rho)$ is an effective skeleton, then 
$$Ext(\fk,L,\rho)\cong\{\alpha\in N_{Gl(\fk)}(\rho(L)): \alpha|_{\fl}=\Ad_{Gl(\fk)}(\alpha)|_{d\rho(\fl)}\},$$
where we denote by $N_{Gl(\fk)}(\rho(L))$ the normalizer of $\rho(L)$ in $Gl(\fk)$.

In particular, $Ext(\fk,L,\rho)$ is linear algebraic Lie group with Lie algebra
$$\frak{ext}(\fk,L,\rho)\cong \{\alpha\in \frak{gl}(\fk): \exp(\alpha)\in Ext(\fk,L,\rho)\}.$$
\end{prop}
\begin{proof}
If we consider the map $pr_1(Ext(\fk,L,\rho))\subset Gl(\fk)$ given by projection to the first factor, then $Ker(pr_1)$ consist by extensions $(\id,i)$ for automorphisms $i$ of $L$ such that $di=\id$. On the other hand, the map $\rho$ is injective, because $Ker(\rho)$ is by definition contained in kernel of $(\fk,L,\rho)$. Thus the property (2) of morphisms of skeletons imply that $i(l)=\rho^{-1}(\alpha\circ \rho(l)\circ \alpha^{-1})$ holds for all $l\in L$ and thus $i=\id$. Therefore $pr_1$ is injective and the set $\{\alpha\in N_{Gl(\fk)}(\rho(L)): \alpha|_{\fl}=\Ad_{Gl(\fk)}(\alpha)|_{d\rho(\fl)}\}$ precisely characterizes the image of $pr_1$, because if $\alpha|_{\fl}=\Ad_{Gl(\fk)}(\alpha)|_{d\rho(\fl)}$, then there is unique Lie algebra automorphisms $i(l):=\rho^{-1}(\alpha\circ \rho(l)\circ \alpha^{-1})$ such that $(\alpha,i)$ is extension. In particular, the last claim is clear.
\end{proof}

Since it is simpler to compute the normalizers of linear spaces then subgroups, we can compute the Lie groups of infinitesimal versions of extensions, which are defined as
$$IExt(\fk,\fl,d\rho):=\{\alpha\in N_{Gl(\fk)}(d\rho(\fl)): \alpha|_\fl=\Ad_{Gl(\fk)}(\alpha)|_{d\rho(\fl)}\}$$
and their Lie algebras
$$\frak{iext}(\fk,\fl,d\rho):=\{\alpha\in \frak{gl}(\fk): [\alpha,d\rho(\fl)]\subset d\rho(\fl), \alpha|_\fl=\ad_{\frak{gl}(\fk)}(\alpha)|_{d\rho(\fl)}\},$$
where $[,]$ is the bracket in $\frak{gl}(\fk).$

The following Lemma follows from standard relations between the Lie groups and Lie algebras.

\begin{lem}
It holds:
\begin{enumerate}
\item The Lie groups $IExt(\fk,\fl,d\rho)$ are linear algebraic Lie groups and $\frak{iext}(\fk,\fl,d\rho)$ are linear algebraic Lie algebras.

\item $Ext(\fk,L,\rho)\subset IExt(\fk,\fl,d\rho)$ and $\frak{ext}(\fk,L,\rho)\subset \frak{iext}(\fk,\fl,d\rho)$.

\item The element $\alpha\in IExt(\fk,\fl,d\rho)$ is element of $Ext(\fk,L,\rho)$ if and only if conjugation by $\alpha$ permutes the connected components of $\rho(L)$.

\item The element $\alpha\in \frak{iext}(\fk,\fl,d\rho)$ is element of $\frak{ext}(\fk,L,\rho)$ if and only if $\exp(\alpha)$ permutes the connected components of $\rho(L)$.
\end{enumerate}
\end{lem}
\end{enumerate}

\section{The categories of Cartan geometries with morphisms induced by extension functors}\label{sec2}

Let us start by checking that the definitions \ref{symdef0} and \ref{symdef2} of morphisms in the category $\mathcal{C}$ of Cartan geometries are equivalent.

\begin{prop}\label{altdef}
Let $(\alpha,i)$ be an extension of the skeleton $(\fg,H,\sigma)$ to the skeleton $(\fk,L,\rho)$, let $(\ga\to I,\gamma)$ be a Cartan geometry of type $(\fg,H,\sigma)$ and let $(\ba\to M,\om)$ be a Cartan geometry of type $(\fk,L,\rho)$. Then:

\begin{enumerate}
\item If $\Phi\circ \mathcal{F}_{(\alpha,i)}$ is a morphism from $(\ga\to I,\gamma)$ to $(\ba\to M,\om)$ in the category $\mathcal{C}$ and $\iota$ is the natural inclusion of $\ga\to \mathcal{F}_{(\alpha,i)}(\ga)$, then $$\Phi\circ \iota: \ga\to \ba$$ is a $(\alpha,i)$--morphism from $(\ga\to I,\gamma)$ to $(\ba\to M,\om)$.
\item If $\Phi$ is a $(\alpha,i)$--morphism from $(\ga\to I,\gamma)$ to $(\ba\to M,\om)$ and we define $$\tilde \Phi(r^l\circ \iota(u)):=r^{l}\circ \iota(\Phi(u)): \mathcal{F}_{(\alpha,i)}\ga\to \ba$$ for $u\in \ga$ and $l\in L$, then $\tilde \Phi\circ \mathcal{F}_{(\alpha,i)}$ is a morphism from $(\ga\to I,\gamma)$ to $(\ba\to M,\om)$ in the category $\mathcal{C}$.
\end{enumerate}

Moreover, 

\begin{itemize}
\item If $i$ is injective, then the extension functor $\mathcal{F}_{(\alpha,i)}$ is faithful.
\item If $i$ is bijective, then the extension functor $$\mathcal{F}_{(\alpha^{-1},i^{-1})}:\mathcal{C}^{(\fk,L,\rho)}\to \mathcal{C}^{(\fk,L,\rho)}$$ is faithful, too, and $\mathcal{C}^{(\fk,L,\rho)}$ and $\mathcal{C}^{(\fg,H,\sigma)}$ are equivalent categories.
\end{itemize}
In particular, the isomorphisms in the category $\mathcal{C}$ are in bijective correspondence with the $(\alpha,i)$--isomorphisms.
\end{prop}
\begin{proof}
If $\Phi: \mathcal{F}_{(\alpha,i)}(\ga\to I,\gamma)\to (\ba\to M,\om)$ is a morphism of Cartan geometries of type $(\fk,L,\rho)$, then $\Phi\circ \iota$ satisfies $\Phi\circ \iota(r^h\circ u)=r^{i(h)}\circ\Phi(\iota(u))$ for all $h\in H$ and $(\Phi\circ \iota)^*(\om)=\iota^*(\mathcal{F}_{(\alpha,i)}(\gamma))=\alpha\circ \gamma$, i.e., $\Phi\circ \iota$ is a $(\alpha,i)$--morphism. 

Conversely, if $\Phi$ is a $(\alpha,i)$--morphism from $(\ga\to I,\gamma)$ to $(\ba\to M,\om)$, then $\tilde \Phi(r^{i(h)^{-1}l} \circ \iota(r^h\circ u))=r^{i(h)^{-1}l}\circ\Phi(r^{h}\circ u)=r^{l}\circ\Phi(u)$ holds for all $h\in H$ and $l\in L$. Thus $\tilde \Phi$ is a well--defined $L$--bundle morphism and $\tilde \Phi\circ \iota=\Phi$ holds by the construction. Therefore $\iota^*\circ (\tilde \Phi^*(\om))=\Phi^*(\om)=\alpha\circ\gamma$ and thus $\tilde \Phi^*(\om)=\mathcal{F}_{(\alpha,i)}(\gamma)$ follows from the uniqueness part of Theorem \ref{cartoprinc}. Therefore $\tilde \Phi\circ \mathcal{F}_{(\alpha,i)}$ is a morphism from $(\ga\to I,\gamma)$ to $(\ba\to M,\om)$ in the category $\mathcal{C}$.

One of the consequences of Theorem \ref{Liegp} is that we can decide the equality of morphisms of Cartan geometries in $\mathcal{C}^{(\fg,H,\sigma)}$ or $\mathcal{C}^{(\fk,L,\rho)}$ by equality of images of single point. In particular, if morphisms $\Phi,\Phi'\in \mathcal{C}^{(\fg,H,\sigma)}$ map chosen point $v\in \ga$ on the points $u\in \ga$ and $r^h\circ u'\in \ga$ for $h\in H$, then $\mathcal{F}_{(\alpha,i)}(\Phi),\mathcal{F}_{(\alpha,i)}(\Phi')\in \mathcal{C}^{(\fk,L,\rho)}$ map the point $\iota(v)\in \mathcal{F}_{(\alpha,i)}(\ga)$ on the points $\iota(u)$ and $r^{i(h)}\circ \iota(u')$ and thus if $i$ is injective, then the functor $\mathcal{F}_{(\alpha,i)}$ is faithful. If $i$ is bijective, then $(\alpha^{-1},i^{-1})$ is a well--defined extension of $(\fk,L,\rho)$ to $(\fg,H,\sigma)$ and the second claim follows, because the composition of $\mathcal{F}_{(\alpha,i)}$ and $\mathcal{F}_{(\alpha^{-1},i^{-1})}$ is the identity functor.
\end{proof}

Let us start with the construction of the functor $\mathcal{F}_{\mathcal{S}}$ that will allow us to characterize and compute the $\mathcal{S}$--automorphism groups. Firstly, we will need a general construction of extension functors between the categories $\mathcal{SC}^{(\fk,L,\rho)}$ and $\mathcal{S'C}^{(\fk',L',\rho')}$.

\begin{prop}\label{extinsc}
Let $(\alpha,i)$ be an extension of $(\fk,L,\rho)$ to $(\fk',L',\rho')$ and let $F: \mathcal{S}\to \mathcal{S'}$ be a group homomorphism satisfying $F(\beta,j)\circ (\alpha,i)=(\alpha,i)\circ (\beta,j)$ for all $(\beta,j)\in \mathcal{S}$. Then there is an extension functor $\mathcal{F}_{(\alpha,i),F}$ from $\mathcal{SC}^{(\fk,L,\rho)}$ to $\mathcal{S'C}^{(\fk',L',\rho')}$ given as
$$\mathcal{F}_{(\alpha,i),F}(\ba\to M,\om):=\mathcal{F}_{(\alpha,i)}(\ba\to M,\om)$$
for $(\ba\to M,\om)\in \mathcal{SC}^{(\fk,L,\rho)}$ and
$$\mathcal{F}_{(\alpha,i),F}(\Phi\circ \mathcal{F}_{(\beta,j)}):=\mathcal{F}_{(\alpha,i)}(\Phi)\circ \mathcal{F}_{F(\beta,j)}$$
for the morphism $\Phi\circ\mathcal{F}_{(\beta,j)}$ in the category $\mathcal{SC}^{(\fk,L,\rho)}$.
\end{prop}
\begin{proof}
Since $\mathcal{F}_{(\alpha,i)}\circ \mathcal{F}_{(\beta,j)}(\ba\to M,\om)=\mathcal{F}_{F(\beta,j)}\circ \mathcal{F}_{(\alpha,i)}(\ba\to M,\om)$ holds by the assumption, it is clear that $\mathcal{F}_{(\alpha,i),F}(\Phi\circ \mathcal{F}_{(\beta,j)})$ is a well--defined morphism in the category $\mathcal{S'C}^{(\fk',L',\rho')}$. Thus to show that $\mathcal{F}_{(\alpha,i),F}$ is a functor it suffices to show that $\mathcal{F}_{(\alpha,i),F}$ is compatible with the compositions. Since $F(\beta,j)\circ (\alpha,i)=(\alpha,i)\circ (\beta,j)$, we can compute 
\begin{align*} 
&\mathcal{F}_{(\alpha,i),F}(\Phi\circ \mathcal{F}_{(\beta,j)}\circ \bar \Phi\circ \mathcal{F}_{(\bar \beta,\bar j)})
=\mathcal{F}_{(\alpha,i)}(\Phi\circ  \mathcal{F}_{(\beta,j)}(\bar \Phi))\circ \mathcal{F}_{F(\beta,j)\circ F(\bar \beta,\bar j)}\\
&=\mathcal{F}_{(\alpha,i)}(\Phi)\circ \mathcal{F}_{F(\beta,j)\circ (\alpha,i)}(\bar \Phi)\circ \mathcal{F}_{F(\beta,j)}\circ \mathcal{F}_{F(\bar \beta,\bar j)}\\
&=\mathcal{F}_{(\alpha,i),F}(\Phi\circ \mathcal{F}_{(\beta,j)})\circ \mathcal{F}_{(\alpha,i),F}(\bar \Phi\circ \mathcal{F}_{(\bar \beta,\bar j)})).
\end{align*}
\end{proof}

Let us construct the skeleton for the target of the functor $\mathcal{F}_{\mathcal{S}}$. Suppose $\mathcal{S}$ is a subgroup of $Ext(\fk,L,\rho)$ and a Lie subgroup of $Gl(\fk)\times \Aut(L)$. We will always denote by $\fs$ the Lie algebra of $\mathcal{S}$. Then the projection to second component $pr_2$ is a Lie group homomorphism $\mathcal{S}\to \Aut(L)$ and we can form a semidirect product $L\rtimes_{pr_2}\mathcal{S}$ of Lie groups. The Lie algebra $\fl\oplus \fs$ of $L\rtimes_{pr_2}\mathcal{S}$ has the Lie bracket
$$[X+Z,Y+W]=\ad_\fl(X)(Y)+dpr_1(Z)(Y)-dpr_1(W)(X)+\ad_{\fs}(Z)(W)$$
for $X,Y\in \fl$ and $Z,W\in \fs$, where we use the fact that the elements $dpr_1(Z),dpr_1(W)\in \frak{gl}(\fk)$ preserve $\fl$. Moreover, if $\mathcal{S}$ is a Lie subgroup of the group $\{(\alpha,\id)\in Ext(\fk,L,\rho)\}$, then $L\rtimes_{pr_2}\mathcal{S}=L\times \mathcal{S}$ is a direct product of Lie groups. 

We can extend the action $\rho$ by the adjoint representation of $L\rtimes_{pr_2}\mathcal{S}$ to obtain skeleton $(\fk\oplus \frak{s},L\rtimes_{pr_2} \mathcal{S},\rho_{\mathcal{S}})$, where $$\rho_{\mathcal{S}}(l,(\alpha,i))(Y+W):=\rho(l)(\alpha(Y))+\Ad_{L\rtimes_{pr_2}\mathcal{S}}(l,(\alpha,i))(W)$$ for $(l,(\alpha,i))\in L\rtimes_{pr_2}\mathcal{S}$, $Y\in \fk$ and $W\in \frak{s}$. It holds $$d\rho_{\mathcal{S}}(X+Z)(Y+W)=d\rho(X)(Y)+dpr_1(Z)(Y)-dpr_1(W)(X)+\ad_{\fs}(Z)(W)$$ for $X\in \fl$, $Y\in \fk$ and $Z,W\in \fs$. Let us point out that in general, there is no compatible Lie bracket of $\fk$ such that $dpr_1(\fs)\subset \frak{gl}(\fk)$ acts by derivations on $\fk$.

As a next step of the construction of the functor $\mathcal{F}_{\mathcal{S}}$, we construct an extension functor according to the Proposition \ref{extinsc}.

\begin{prop}\label{extdescr}
Suppose $\mathcal{S}$ is a subgroup of $Ext(\fk,L,\rho)$ and a Lie subgroup of $Gl(\fk)\times \Aut(L)$. Then the natural inclusion $\iota: L\to L\rtimes_{pr_2}\mathcal{S}$ and $\alpha_\iota: \fk\to \fk\oplus \fs$ provide an injective extension $(\alpha_\iota,\iota)$ of $(\fk,L,\rho)$ to $(\fk\oplus \frak{s},L\rtimes_{pr_2}\mathcal{S},\rho_{\mathcal{S}})$ and $F: \mathcal{S}\to Ext(\fk\oplus \frak{s},L\rtimes_{pr_2}\mathcal{S},\rho_{\mathcal{S}})$ given as $$(\beta,j)\mapsto ((\beta,\Ad_{\mathcal{S}}(\beta,j),(j,\conj(\beta,j)))$$ is an injective homomorphism of Lie groups such that $F(\beta,j)\circ (\alpha_\iota,\iota)=(\alpha_\iota,\iota)\circ (\beta,j)$ holds for all $(\beta,j)\in \mathcal{S}$. In particular, $\mathcal{F}_{(\alpha_\iota,\iota),F}$ is a faithful extension functor from $\mathcal{SC}^{(\fk,L,\rho)}$ to $Ext(\fk\oplus \frak{s},L\rtimes_{pr_2}\mathcal{S},\rho_{\mathcal{S}})\mathcal{C}^{(\fk\oplus \frak{s},L\rtimes_{pr_2}\mathcal{S},\rho_{\mathcal{S}})}$.
\end{prop}
\begin{proof}
It is obvious that $(\alpha_\iota,\iota)$ is injective extension by the construction of the skeleton $(\fk\oplus \frak{s},L\rtimes_{pr_2}\mathcal{S},\rho_{\mathcal{S}})$. Further, it is simple computation to check that $(\beta,j)\mapsto ((\beta,\Ad_{\mathcal{S}}(\beta,j)),(j,\conj(\beta,j)))$ is element of $Ext(\fk\oplus \frak{s},L\rtimes_{pr_2}\mathcal{S},\rho_{\mathcal{S}})$ for all $(\beta,j)\in \mathcal{S}$ and $F(\beta,j)\circ F(\bar \beta,\bar j)=((\beta,\Ad_{\mathcal{S}}(\beta,j))\circ(\bar \beta,\Ad_{\mathcal{S}}(\bar \beta,\bar j)),(j,\conj(\beta,j))\circ(\bar j,\conj(\bar \beta,\bar j)))=((\beta\circ\bar \beta,\Ad_{\mathcal{S}}(\beta,j)\circ\Ad_{\mathcal{S}}(\bar \beta,\bar j)),(j\circ \bar j,\conj(\beta,j)\circ \conj(\bar \beta,\bar j)))=F((\beta,j)\circ(\bar \beta,\bar j))$ holds for all $(\beta,j),(\bar \beta,\bar j)\in \mathcal{S}$. Therefore $F$ is a injective homomorphism of Lie groups and $F(\beta,j)\circ (\alpha_\iota,\iota)=(\alpha_\iota,\iota)\circ (\beta,j)$ holds, because the maps $(\beta,\Ad_{\mathcal{S}}(\beta,j))$ and $(j,\conj(\beta,j))$ preserve the decompositions $\fk\oplus \fs$ and $L\rtimes_{pr_2}\mathcal{S}$, respectively. Thus the remaining claims follow from the Proposition \ref{extinsc}.
\end{proof}

As the final step of  the construction of $\mathcal{F}_{\mathcal{S}}$, we modify the functor $\mathcal{F}_{(\alpha_\iota,\iota),F}$ from the Proposition \ref{extdescr}. For this, we observe an existence of a particular class of morphisms in the category $Ext(\fk\oplus \frak{s},L\rtimes_{pr_2}\mathcal{S},\rho_{\mathcal{S}})\mathcal{C}^{(\fk\oplus \frak{s},L\rtimes_{pr_2}\mathcal{S},\rho_{\mathcal{S}})}$.

\begin{lem}\label{nonefext}
The right multiplication by an element $(e,(\beta,j))\in L\rtimes_{pr_2}\mathcal{S}$ is an $((\beta^{-1},\Ad_{\mathcal{S}}(\beta,j)^{-1}),(j^{-1},\conj(\beta,j)^{-1}))$--automorphism on each Cartan geometry of type $(\fk\oplus \frak{s},L\rtimes_{pr_2}\mathcal{S},\rho_{\mathcal{S}})$. In particular, $$r^{(e,(\beta,j))}\circ \mathcal{F}_{(\alpha_\iota,\iota),F}(\Phi)\in \mathcal{C}^{(\fk\oplus \frak{s},L\rtimes_{pr_2}\mathcal{S},\rho_{\mathcal{S}})}$$ holds for all $(\beta,j)\in \mathcal{S}$ and each $(\beta,j)$--morphism $\Phi$ in $\mathcal{SC}^{(\fk,L,\rho)}$.
\end{lem}
\begin{proof}
The first claim follows from the definition of Cartan geometries. Then the second claim follow from the Proposition \ref{extdescr}.
\end{proof}

Since $\mathcal{F}_{(\alpha_\iota,\iota),F}(\Phi)\circ r^{(e,(\bar \beta,\bar j))}=r^{(e,(\beta,j)\circ (\bar \beta,\bar j)\circ (\beta,j)^{-1})}\circ \mathcal{F}_{(\alpha_\iota,\iota),F}(\Phi)$ holds for all $(\beta,j),(\bar \beta,\bar j)\in \mathcal{S}$ and each $(\beta,j)$--morphism $\Phi$ in $\mathcal{SC}^{(\fk,L,\rho)}$, the formula $r^{(e,(\beta,j))}\circ \mathcal{F}_{(\alpha_\iota,\iota),F}(\Phi)$ is compatible with the composition of $(\beta,j)$--morphisms with $(\beta',j')$--morphisms. Therefore we can define the following functor.

\begin{def*}\label{defskelpr}
We denote by $\mathcal{F}_{\mathcal{S}}$ the functor $\mathcal{SC}^{(\fk,L,\rho)}$ to  $\mathcal{C}^{(\fk\oplus \frak{s},L\rtimes_{pr_2}\mathcal{S},\rho_{\mathcal{S}})}$ defined as
$$\mathcal{F}_{\mathcal{S}}(\ba\to M,\om):=\mathcal{F}_{(\alpha_\iota,\iota)}(\ba\to M,\om)$$
for $(\ba\to M,\om)\in \mathcal{SC}^{(\fk,L,\rho)}$ and
$$\mathcal{F}_{\mathcal{S}}(\Phi):=r^{(e,(\beta,j))}\circ \mathcal{F}_{(\alpha_\iota,\iota),F}(\Phi)$$
for $(\beta,j)$--morphisms $\Phi$ in $\mathcal{SC}^{(\fk,L,\rho)}$.
\end{def*}

Thus it remains to show that the functor $\mathcal{F}_{\mathcal{S}}$ is full in order to relate the properties of categories $\mathcal{SC}^{(\fk,L,\rho)}$ and $\mathcal{C}^{(\fk\oplus \frak{s},L\rtimes_{pr_2}\mathcal{S},\rho_{\mathcal{S}})}$.

\begin{thm}\label{main1}
Suppose $\mathcal{S}$ is a subgroup of $Ext(\fk,L,\rho)$ and a Lie subgroup of $Gl(\fk)\times \Aut(L)$. Let $\Phi$ be a $(\beta,j)$--morphism from $(\ba\to M,\om)$ to $(\ba'\to M',\om')$ in $\mathcal{SC}^{(\fk,L,\rho)}$. Then $$\mathcal{F}_{\mathcal{S}}(\Phi)(r^{(l,(\bar \beta,\bar j))}\circ \iota(u))=r^{(j(l),(\beta,j)\circ(\bar \beta,\bar j))}\circ \iota(\Phi(u))$$
holds for all $u\in \ba$ and all $(l,(\bar \beta,\bar j))\in L\rtimes_{pr_2}\mathcal{S}$ and the functor $\mathcal{F}_{\mathcal{S}}$ is full.

In particular, $$\Aut(\mathcal{F}_{\mathcal{S}}(\ba\to M,\om))=\mathcal{F}_{\mathcal{S}}(\mathcal{S}\Aut(\ba,\om))$$ is a Lie group of dimension at most $dim(\Aut(\ba,\om))+dim(\frak{s})$.
\end{thm}
\begin{proof}
Let $u\in \ba$ and $(l,(\bar \beta,\bar j))\in L\rtimes_{pr_2}\mathcal{S}$. Then \begin{align*}
\mathcal{F}_{\mathcal{S}}(\Phi)(r^{(l,(\bar \beta,\bar j))}\circ \iota(u))&=r^{(e,(\beta,j))}\circ\mathcal{F}_{(\alpha_\iota,\iota),F}(\Phi)(r^{(l,(\bar \beta,\bar j))}\circ \iota(u))\\
&=r^{(e,(\beta,j))}\circ r^{(j(l),\conj(\beta,j)(\bar \beta,\bar j))}\circ \mathcal{F}_{(\alpha_\iota,\iota)}(\Phi)(\iota(u))\\
&=r^{(j(l),(\beta,j)\circ(\bar \beta,\bar j))}\circ \iota(\Phi(u))
\end{align*}
holds.

Conversely, suppose $\tilde \Phi$ is a morphism in $\mathcal{C}^{(\fk\oplus \frak{s},L\rtimes_{pr_2}\mathcal{S},\rho_{\mathcal{S}})}$ from $\mathcal{F}_{\mathcal{S}}(\ba\to M,\om)$ to $\mathcal{F}_{\mathcal{S}}(\ba'\to M',\om')$ satisfying $\tilde \Phi(\iota(u))=r^{(e,(\beta,j))}\circ \iota(\Phi(u))$ for some fixed $u\in \ba$. Then \begin{align*}
\tilde \Phi(\iota(Fl_t^{\om^{-1}(X)}(u)))&=\tilde \Phi(Fl_t^{(\mathcal{F}_{\mathcal{S}}(\om))^{-1}(\alpha_\iota(X))}(\iota(u)))=Fl_t^{(\mathcal{F}_{\mathcal{S}}(\om))^{-1}(\alpha_\iota(X))}(r^{(e,(\beta,j))}\circ \iota(\Phi(u)))\\
&=r^{(e,(\beta,j))}\circ \iota(Fl_t^{\om^{-1}(\beta(X))}(\Phi(u)))
\end{align*} 
holds for all $X\in \fk$, because $$Tr^{(e,(\beta,j))}(\mathcal{F}_{\mathcal{S}}(\om))^{-1}(\Ad_{L\rtimes_{pr_2}\mathcal{S}}(e,(\beta,j))(\alpha_\iota(X)))=Tr^{(e,(\beta,j))}\circ T\iota\circ \om^{-1}(\beta(X))$$
holds at all points of $\iota(\ba)$. Since $M$ is connected and $\tilde \Phi$ is $(L\rtimes_{pr_2} \mathcal{S})$--bundle morphism, the extension $(\beta,j)\in \mathcal{S}$ does not depend on the choice of $u\in \ba$ and $$\Phi:=\iota^{-1}\circ r^{(e,(\beta,j)^{-1})}\circ \tilde \Phi|_{\iota(\ba)}$$ 
is a well--defined map $\ba\to\ba'$. By definition, $(\Phi)^*(\om')=\beta\circ \om$ holds and $$\iota(\Phi(r^l\circ u))=r^{(e,(\beta,j)^{-1})}\circ r^{\iota(l)}\circ\tilde \Phi(\iota(u))= r^{\iota(j(l))}\circ r^{(e,(\beta,j)^{-1})}\circ\tilde \Phi(\iota(u))=\iota(r^{j(l)}\circ \Phi(u))$$ holds for all $l\in L$ and $u\in \ba$, thus $\Phi$  is a $(\beta,j)$--morphism from $(\ba\to M,\om)$ to $(\ba'\to M',\om')$. Therefore the functor $\mathcal{F}_{\mathcal{S}}$ is full.

If $(\ba\to M,\om)=(\ba'\to M',\om')$, then the claim $\Aut(\mathcal{F}_{\mathcal{S}}(\ba\to M,\om))\cong \mathcal{S}\Aut(\ba,\om)$ follows and the dimension estimate follows from the Theorem \ref{Liegp}.
\end{proof}

Let us point out that $\mathcal{F}_{\mathcal{S}}(\mathcal{SC}^{(\fk,L,\rho)})$ is only a full subcategory of $\mathcal{C}^{(\fk\oplus \frak{s},L\rtimes_{pr_2}\mathcal{S},\rho_{\mathcal{S}})}$, in general.

As we observed in Lemma \ref{nonefext}, if we consider the right action $r^{l}$ for $l\in L$, then we obtain a trivial $(\rho(l)^{-1},\conj(l)^{-1})$--automorphism of each Cartan geometry of type $(\fk,L,\rho)$. If $(\rho(l)^{-1},\conj(l)^{-1})\in \mathcal{S}$, then from Theorem \ref{main1} follows that
$$\mathcal{F}_{\mathcal{S}}(r^{l})(\iota(u))=r^{(l,(\rho(l)^{-1},\conj(l)^{-1}))}\circ \iota(u)$$ holds for all $u\in \ba.$ Therefore the kernel of $(\fk\oplus\frak{s},L\rtimes_{pr_2} \mathcal{S},\rho_{\mathcal{S}})$ is not trivial, in general.

Let us consider a general situation first and denote by $N$ the kernel of $(\fk,L,\rho)$. Then the projections $\pi_\fn: \fk\to \fk/\fn$ and $\pi_N: L\to L/N$ satisfy the conditions of an extension and thus there is a natural extension functor from $\mathcal{C}^{(\fk,L,\rho)}$ to $\mathcal{C}^{(\fk/\fn,L/N,\rho)}$. Thus from Theorem \ref{rigid} follows that the kernel of $\mathcal{F}_{(\pi_\fn,\pi_N)}$ on automorphisms are the trivial automorphisms in $\mathcal{C}^{(\fk,L,\rho)}$. Let us look on the situation in the categories $\mathcal{S}\mathcal{C}^{(\fk,L,\rho)}$.

\begin{lem}\label{funceff}
There is unique group homomorphisms $F: Ext(\fk,L,\rho)\to Ext(\fk/\fn,L/N,\rho)$ such that $\mathcal{F}_{(\pi_\fn,\pi_N)}\circ \mathcal{F}_{(\beta,j)}=\mathcal{F}_{F(\beta,j)}\circ \mathcal{F}_{(\pi_\fn,\pi_N)}$ holds for all $(\beta,j)\in Ext(\fk,L,\rho)$ and the kernel of $F$ is the normal subgroup $$\{(\beta,j)\in Ext(\fk,L,\rho): (\id-\beta)\fk\subset \fn\}.$$ In particular, we will use uniform notation $$\mathcal{F}_{eff}:=\mathcal{F}_{(\pi_\fn,\pi_N),F}: \mathcal{SC}^{(\fk,L,\rho)}\to F \mathcal{(S)C}^{(\fk/\fn,L/N,\rho)}$$ for all extension functors induced by $(\pi_\fn,\pi_N)$ and $F$.
\end{lem}
\begin{proof}
To prove the first claim, it clearly suffices to show that $\beta(\fn)\subset \fn$ and $j(N)\subset N$ holds for all $(\beta,j)\in Ext(\fk,L,\rho)$, because then $F(\beta,j)=(\pi_\fn,\pi_N)\circ(\beta,j)\circ(\pi_\fn,\pi_N)^{-1}$ is well--defined. It holds $lj(n)l^{-1}=j(j^{-1}(l)nj^{-1}(l^{-1}))\in j(N)$ for $n\in N$ and $l\in L$, thus $j(N)$ is normal subgroup of $L$. Further, $(\id-\rho(j(n))=\beta\circ (\id-\rho(n))\circ \beta^{-1}$ holds for $n\in N$, thus $(\id-\rho(j(n))\fk\subset \beta(\fn)=dj(\fn)$. Thus $j(N)\subset N$ and $\beta(\fn)\subset \fn$ hold due to maximality of the kernel $N$.

Therefore the conditions of the Proposition \ref{extinsc} are satisfied and there are the extension functors $\mathcal{F}_{eff}$.
\end{proof}

We remark that the map $F: Ext(\fk,L,\rho)\to Ext(\fk/\fn,L/N,\rho)$ is not surjective, in general, because the automorphisms $j$ of $L/N$ does not have to be extendable to automorphisms of $L$ and there does not have to be compatible inclusion of $\beta\in Gl(\fk/\fn)$ into $Gl(\fk)$.

Thus we only need to investigate in the situation, when $(\fk,L,\rho)$ is effective, whether there are other trivial $\mathcal{S}$--automorphisms then $r^l$ for $l\in L$ and $(\rho(l)^{-1},\conj(l)^{-1})\in \mathcal{S}$ mentioned above.

\begin{thm}\label{effective}
Suppose that $(\fk,L,\rho)$ is effective and assume that $\mathcal{S}$ is a subgroup of $Ext(\fk,L,\rho)$ and a Lie subgroup of $Gl(\fk)\times \Aut(L)$, and denote by  $N_\mathcal{S}$ the kernel of $(\fk\oplus\frak{s},L\rtimes_{pr_2} \mathcal{S},\rho_{\mathcal{S}})$.

An element $(l,(\beta,j))\in L\rtimes_{pr_2} \mathcal{S}$ is in $N_\mathcal{S}$ if and only if $(\id -\rho(l)\circ \beta)\fk=0$ holds.

In particular, $$N_{\mathcal{S}}=\{(l,(\rho(l^{-1}),\conj(l)^{-1})): l\in L, (\rho(l^{-1}),\conj(l)^{-1})\in \mathcal{S}\}$$ and $$\mathcal{S}\Aut(\ba,\om)/ \{r^l: (\rho(l^{-1}),\conj(l)^{-1})\in \mathcal{S}\} \cong \Aut(\mathcal{F}_{eff}\circ \mathcal{F}_{\mathcal{S}}(\ba\to M,\om)).$$
\end{thm}
\begin{proof}
It holds that $\rho_{\mathcal{S}}(L\rtimes_{pr_2} \mathcal{S})\fk\subset \fk$. Thus $(\id-\rho_{\mathcal{S}}(l,(\beta,j)))\fk\subset \fn_\mathcal{S}\cap \fk$ holds for $(l,(\beta,j))\in N_\mathcal{S}$ and thus $\fn_\mathcal{S}\cap \fk=0$. Therefore $(\id -\rho(l)\circ \beta)\fk=0$ holds for all $(l,(\beta,j))\in N_\mathcal{S}$

Conversely, consider $N':=\{(l,(\beta,j))\in L\rtimes_{pr_2} \mathcal{S}: (\id -\rho(l)\circ \beta)\fk=0\}$. Since the condition $(\id -\rho(l)\circ \beta)\fk=0$ is stable under the conjugation of the element $(l,(\beta,j))$ by elements of $L\rtimes_{pr_2} \mathcal{S}$, the subgroup $N'$ is normal subgroup of $L\rtimes_{pr_2} \mathcal{S}$ containing $N_{\mathcal{S}}$ and we get $N'=N_{\mathcal{S}}$ by the maximality of the kernel.

If $\id=\rho(l)\circ \beta$ holds for $(l,(\beta,j))\in N_{\mathcal{S}}$, then $(\beta,j)=(\rho(l^{-1}),\conj(l)^{-1})$ follows from the Proposition \ref{extcomp} and therefore the $\mathcal{S}$--automorphisms $r^l$ for $l\in L$ and $(\rho(l)^{-1},\conj(l)^{-1})\in \mathcal{S}$ exhaust all trivial automorphisms in $\mathcal{S}\Aut(\ba,\om)$.
\end{proof}

We remark that in order to get a splitting $\Aut(\mathcal{F}_{eff}\circ \mathcal{F}_{\mathcal{S}}(\ba\to M,\om))\to \mathcal{S}\Aut(\ba,\om)$, we need to choose a splitting $\mathcal{S}/(\{(\rho(l),\conj(l)),l\in L \}\cap \mathcal{S})\to \mathcal{S}$.

\section{Computation of the automorphism groups}\label{sec5}

Let us start with a Cartan geometry $(\ba\to M,\om)$ of type $(\fk,L,\rho)$ and suppose $\mathcal{S}$ is a subgroup of $Ext(\fk,L,\rho)$ and a Lie subgroup of $Gl(\fk)\times \Aut(L)$. We can modify the method in \cite[Lemma 1.5.12]{parabook} for the computation of infinitesimal automorphisms and compute the Lie algebra of the Lie group $\mathcal{S}\Aut(\ba,\om)$. We consider the bundle $$\mathcal{A}_{\mathcal{S}}M:=\mathcal{F}_{\mathcal{S}}(\ba)\times_{\rho_{\mathcal{S}}}(\fk\oplus \fs),$$ because there is bijection between sections of $\mathcal{A}_{\mathcal{S}}M$ and $L\rtimes_{pr_2} \mathcal{S}$--invariant vector fields on $\mathcal{F}_{\mathcal{S}}(\ba)$ provided by the Cartan connection $\mathcal{F}_{\mathcal{S}}(\om)$. This means that the infinitesimal automorphisms of $\mathcal{F}_{\mathcal{S}}(\ba\to M,\om)$ can be viewed as sections of $\mathcal{A}_{\mathcal{S}}M$ and the formula in Theorem \ref{Liegp} leads to the following characterization of infinitesimal automorphisms.

\begin{prop}\label{autcomp0}
Let $(\ba\to M,\om)$ be a Cartan geometry of type $(\fk,L,\rho)$ and suppose $\mathcal{S}$ is a subgroup of $Ext(\fk,L,\rho)$ and a Lie subgroup of $Gl(\fk)\times \Aut(L)$. Then the infinitesimal automorphism of $\mathcal{F}_{\mathcal{S}}(\ba\to M,\om)$ are sections of  $\mathcal{A}_{\mathcal{S}}M$ parallel w.r.t. to the unique linear connection $\nabla$ on $\mathcal{A}_{\mathcal{S}}M$ satisfying $$(T\Pi)^*(\nabla s)(\mathcal{F}_{\mathcal{S}}\om)^{-1}(t))=D_{t}s-d(\mathcal{F}_{\mathcal{S}}\om)(\mathcal{F}_{\mathcal{S}}\om)^{-1}(t),(\mathcal{F}_{\mathcal{S}}\om)^{-1}(s))$$ for all sections $t,s$ of $\mathcal{A}_{\mathcal{S}}M$, where $\Pi: \mathcal{A}_{\mathcal{S}}M\to TM$ is the natural projection and the fundamental derivative $D_{t}s(v)$ is given by the directional derivative of $s$ in direction of $(\mathcal{F}_{\mathcal{S}}\om)^{-1}(t(v))$ at $v\in \mathcal{F}_{\mathcal{S}}(\ba)$.
\end{prop}
\begin{proof}
The condition $\mathcal{L}_\xi(\mathcal{F}_{\mathcal{S}}(\om))=0$ from the Theorem \ref{Liegp} is equivalent to the condition that $\xi\cdot \mathcal{F}_{\mathcal{S}}(\om)(\nu)-\mathcal{F}_{\mathcal{S}}(\om)([\xi,\nu])=0$ holds for all $L\rtimes_{pr_2} \mathcal{S}$--invariant vector fields $\nu$ on $\mathcal{F}_{\mathcal{S}}(\ba)$. This can be equivalently written as $D_st = [s, t]$ for sections $t,s$ of $\mathcal{A}_{\mathcal{S}}M$, where $[,]$ is the (geometric) bracket on $\mathcal{A}_{\mathcal{S}}M$ induced by the bracket $[\xi,\nu]$ of the corresponding $L\rtimes_{pr_2} \mathcal{S}$--invariant vector fields $\xi,\nu$ on $\mathcal{F}_{\mathcal{S}}(\ba)$. As in proof of \cite[Corollary 1.5.8]{parabook}, we get the following equality $$[s,t]=D_{s}t-D_{t}s-d(\mathcal{F}_{\mathcal{S}}\om)((\mathcal{F}_{\mathcal{S}}\om)^{-1}(s),(\mathcal{F}_{\mathcal{S}}\om)^{-1}(t))$$
relating the geometric bracket with the fundamental derivative. We remark that the claim of \cite[Corollary 1.5.8]{parabook} is not available, because the curvature and the algebraic bracket on $\mathcal{A}_{\mathcal{S}}M$ are not well--defined in our situation. This allows us to rewrite the formula $D_st = [s, t]$ in the claimed way.

Since $(\mathcal{L}_{\zeta_X}\om)(\xi)=d\om(\zeta_X,\xi)$ holds for all fundamental vector fields $\zeta_X$ for $X\in \fl\oplus \fs$, the linear connection $\nabla$ computing infinitesimal automorphism is uniquely defined by the claimed formula.
\end{proof}

Thus the remaining task for the computation of the Lie algebra of the Lie group $\mathcal{S}\Aut(\ba,\om)$ is to check, which sections of $\mathcal{A}_{\mathcal{S}}M$ parallel w.r.t. to the linear connection from the Proposition \ref{autcomp0} correspond to complete vector fields.

Let us look on the case, when $(\alpha,i)$ is an extension of $(\fg,H,\Ad_G|_H)$ corresponding to Klein geometry $(G,H)$ to $(\fk,L,\rho)$ and we want to compute the Lie algebra of infinitesimal automorphisms of the Cartan geometry $\mathcal{F}_{\mathcal{S}}\circ  \mathcal{F}_{(\alpha,i)}(\ga\to M,\om)$ for the Cartan geometry $(\ga\to M,\om)$ of type $(\fg,H,\Ad_G|_H)$. Let us denote by $\kappa:\ga\to \wedge^2(\fg/\fh)^*\otimes \fg$ the curvature of the Cartan geometry $(\ga\to M,\om)$.

The identification $$\mathcal{A}_{\mathcal{S}}M=\ga\times_{\rho_{\mathcal{S}}\circ i(H)}(\alpha(\fg)+\fl\oplus \fs)$$ follows from the construction of the extension functor $\mathcal{F}_{\mathcal{S}}\circ \mathcal{F}_{(\alpha,i)}$ and we will write $$t=\alpha(t_\fg)+t_\fl+t_\fs$$ for the decomposition of the $\rho_{\mathcal{S}}\circ i(H)$--equivariant function $t: \ga\to \fk\oplus\fs$ to the functions $t_\fg: \ga\to \fg$, $t_\fl: \ga\to \fl$ and $t_\fs: \ga\to \fs$. The decomposition $t=\alpha(t_\fg)+t_\fl+t_\fs$ depends on a choice, because $\fk=\alpha(\fg)+\fl$ is not direct sum, however, the formulas that will appear in the rest of this section do not depend on the particular choice of the decomposition. Further, the function $t_\fs$ is constant along fibers, because $(\id-\rho_\mathcal{S}((i(h),(\id,\id)))t_\fs(u)\in \fk$ holds for all $h\in H$ and all $u\in \ga$. In particular, the functions $\alpha(t_\fg)+t_\fl$ and $t_\fs$ are not sections of $\mathcal{A}_{\mathcal{S}}M$, in general.

On the other hand, 
\begin{align*}&d(\mathcal{F}_{\mathcal{S}}\circ \mathcal{F}_{(\alpha,i)}\om)((\mathcal{F}_{\mathcal{S}}\circ \mathcal{F}_{(\alpha,i)}\om)^{-1}(\alpha(t_\fg)),(\mathcal{F}_{\mathcal{S}}\circ \mathcal{F}_{(\alpha,i)}\om)^{-1}(s))=\alpha(d\om(t_\fg,s_\fg))+\\
&d\rho_{\mathcal{S}}(s_\fl+s_\fs)(\alpha(t_\fg))=\alpha(\kappa(t_\fg,s_\fg)+\ad_\fg(s_\fg)(t_\fg))+d\rho_{\mathcal{S}}(s_\fl+s_\fs)(\alpha(t_\fg))
\end{align*}
holds for all $H$--equivariant functions $t_\fg: \ga\to \fg$ and all sections $s$ of $\mathcal{A}_{\mathcal{S}}M$. This provides a different connection on $\mathcal{A}_{\mathcal{S}}M$ that is related to the computation of infinitesimal automorphisms.

\begin{thm}\label{autcomp}
Let $(\ga\to M,\om)$ be a Cartan geometry of type $(\fg,H,\Ad_G|_H)$, let $(\alpha,i)$ be an extension of $(\fg,H,\Ad_G|_H)$ to $(\fk,L,\rho)$ and suppose $\mathcal{S}$ is a subgroup of $Ext(\fk,L,\rho)$ and a Lie subgroup of $Gl(\fk)\times \Aut(L)$. Then the map $A: \fg\to \frak{gl}(\fk\oplus\fs)$ defined by the formula 
$$A(t_\fg)(s):=-\alpha(\ad_\fg(s_\fg)(t_\fg))-d\rho_{\mathcal{S}}(s_\fl+s_\fs)(\alpha(t_\fg))$$
provides a morphism $(A,\rho_\mathcal{S}\circ i)$ of the skeleton $(\fg,H,\Ad_G|_H)$ to the skeleton $(\frak{gl}(\fk\oplus\fs),Gl(\fk\oplus\fs),\Ad_{Gl(\fk\oplus\fs)})$. Therefore the infinitesimal automorphisms of $\mathcal{F}_{\mathcal{S}}\circ  \mathcal{F}_{(\alpha,i)}(\ga\to M,\om)$ are sections $s$ of $\mathcal{A}_{\mathcal{S}}M$ such that $$(\nabla^{A} s)_{\fk}=-\alpha\circ ins_{s_\fg} \circ \kappa$$
and
$$(\nabla^{A} s)_{\fs}=0$$ hold for the unique linear connection $\nabla^{A}$ obtained for morphism $(A,\rho_\mathcal{S}\circ i)$ in Theorem \ref{cartoprinc}, where $ins_{s_\fk}$ is the insertion map to the first slot of the map $\kappa$. In particular, $s_\fs:\ga\to \fs$ is constant for any infinitesimal automorphism $s$.
\end{thm}
\begin{rem}
Let us point out that the Theorem can be used in the case when $\mathcal{S}=\{(\id,\id)\}$ to compute the infinitesimal automorphisms of the Cartan geometry $\mathcal{F}_{(\alpha,i)}(\ga\to M,\om)$. Similarly, the Theorem can be used in the case when $(\fg,H,\Ad_G|_H)=(\fk,L,\rho)$ to compute the Lie algebra of $\mathcal{S}$--automorphism group of $(\ga\to M,\om)$.
\end{rem}
\begin{proof}
We have to check that the map $A$ satisfies that $A|_\fh=d\rho_\mathcal{S}\circ di$ and $A(\Ad_G(h)(X))=\rho_\mathcal{S}(i(h))\circ A(X) \circ \rho_\mathcal{S}(i(h))^{-1}$ hold for all $h\in H$ and $X\in \fg$. But by definition it holds $$A(X)=d\rho_{\mathcal{S}}(di(X))$$ for $X\in \fh$ and we compute $$A(\Ad_G(h)(X))(Y)=\alpha(\ad_\fh(\Ad_G(h)(X))(Y))=\Ad_G(h)\circ A(X)\circ \Ad_G(h)^{-1}(Y)$$ for $X,Y\in \fg$ and $h\in H$ and \begin{align*}
A(\Ad_G(h)(X))(W)&=-d\rho_\mathcal{S}(W)(\alpha(\Ad_G(h)(X)))\\
&=-d\rho_\mathcal{S}(W)\circ \rho_\mathcal{S}(i(h),(\id,\id))\circ\alpha(X)\\&=-\rho_\mathcal{S}(i(h),(\id,\id))d\rho_\mathcal{S}(\rho_\mathcal{S}(i(h),(\id,\id))^{-1}(W))\circ \alpha(X)\\&=\rho_\mathcal{S}(i(h),(\id,\id))\circ A(X)\circ \rho_\mathcal{S}(i(h),(\id,\id))^{-1}(W)
\end{align*} for $X\in \fg$, $W\in \fs$ and $h\in H$.

Then we conclude from the Proposition \ref{autcomp0} that we want to find section $s$ of $\mathcal{A}_{\mathcal{S}}M$ such that \begin{align*}(T\pi)^*(\nabla^As)(\om^{-1}(t_\fg))&=D_{\alpha(t_\fg)}s+A(t_\fg)(s)\\
&=\alpha(\kappa(t_\fg,s_\fg)+\ad_\fg(s_\fg)(t_\fg))+d\rho_{\mathcal{S}}(s_\fl+s_\fs)(\alpha(t_\fg))+A(t_\fg)(s)\\
&=\alpha(\kappa(t_\fg,s_\fg))
\end{align*} holds for all $H$--equivariant functions $t_\fg: \ga\to \fg$ due to the above formula for $$d(\mathcal{F}_{\mathcal{S}}\circ \mathcal{F}_{(\alpha,i)}\om)((\mathcal{F}_{\mathcal{S}}\circ \mathcal{F}_{(\alpha,i)}\om)^{-1}(\alpha(t_\fg)),(\mathcal{F}_{\mathcal{S}}\circ \mathcal{F}_{(\alpha,i)}\om)^{-1}(s))$$.

Since $(T\pi)^*(\nabla^As_\fs)(\om^{-1}(t_\fg))_{\fs}=(D_{\alpha(t_\fg)}s_\fs)_\fs=0$ holds for all infinitesimal automorphisms $s$, the function $s_\fs$ is constant along flows of constant vector fields in $\ga$ and constant along fibers and thus the last claim follows.
\end{proof}

Since the functions $s_{\fs}:\ga\to \fs$ are constant for the complete infinitesimal automorphisms $s$ of $\mathcal{F}_{\mathcal{S}}\circ \mathcal{F}_{(\alpha,i)}(\ga\to M,\om)$, the corresponding one parameter subgroups $\Phi_w$ satisfy $(\Phi_w)^*( \mathcal{F}_{(\alpha,i)}(\om))=\exp(w\cdot d\rho_{\mathcal{S}}(s_{\fs}))\circ  \mathcal{F}_{(\alpha,i)}(\om)$ for $w\in \mathbb{R}$.

In particular, the Theorem \ref{autcomp} can be used on the homogeneous Cartan geometry $\mathcal{F}_{(\alpha,i)}(G\to G/H,\om_G)$ of type $(\fk,L,\rho)$ and we obtain the following result.

\begin{thm}\label{homcomp}
Let $(\alpha,i)$ be an extension of $(\fg,H,\Ad_G|_H)$ to $(\fk,L,\rho)$ and suppose $\mathcal{S}$ is a subgroup of $Ext(\fk,L,\rho)$ and a Lie subgroup of $Gl(\fk)\times \Aut(L)$. Then the connection $\nabla^A$ defined in Theorem \ref{autcomp} is $G$--invariant connection on $\mathcal{A}_\mathcal{S}(G/H)$ and the infinitesimal automorphisms of $\mathcal{F}_{\mathcal{S}}\circ\mathcal{F}_{(\alpha,i)}(G\to G/H,\om_G)$ are sections of $\mathcal{A}_\mathcal{S}(G/H)$ parallel w.r.t. to $\nabla^A$.

In particular, the space of infinitesimal automorphisms on $\mathcal{F}_{\mathcal{S}}\circ \mathcal{F}_{(\alpha,i)}(K\to K/H,\om_K)$ is isomorphic to set of all $X\in \fk\oplus\frak{s}$ such that $$[A(Z_{k+2}),\dots[A(Z_3), A(Z_1)\circ A(Z_2)- A(Z_2)\circ  A(Z_1)-A(\ad_\fg(Z_1)Z_2)]\dots](X)=0$$ holds for all $k$ and $Z_1,\dots, Z_k\in \fg$, where $[,]$ is the Lie bracket in $\frak{gl}(\fk\oplus \fs)$.
\end{thm}
\begin{proof}
Since the map $\om_G\mapsto \nabla^A$ in the Theorem \ref{cartoprinc} is a functor, the first claim is clear. Since the Cartan geometry $(K\to K/H,\om_K)$ is flat, the second claim follows from the Theorem \ref{autcomp}.

It is well known (c.f. \cite[Theorem II.11.8]{KoNo}) that the claimed elements generate the holonomy algebra of the $G$--invariant connection $\nabla^A$ given by the map $A$ and that the parallel local sections correspond to elements of $X\in \fk\oplus\frak{s}$, which are annihilated by the holonomy algebra.
\end{proof}

To obtain a global result, one has to check which solutions $X\in \fk\oplus\frak{s}$ satisfying the conditions of the above Theorem sit not only in a representation of $\fg$, but also in a representation of $G$.

In the special case, when $(\fg,H,\Ad_G|_H)=(\fk,L,\rho)$, we know that all $X\in\fg=\fk$ correspond to the infinitesimal automorphisms. Therefore it is enough to ask, which elements of $X\in \fs$ correspond to infinitesimal automorphisms of $\mathcal{F}_{\mathcal{S}}(K\to K/H,\om_K)$. Since $A(t_\fg)(X)=-d\rho_{\mathcal{S}}(X)(t_\fg)\subset \fg$ holds in this case, we obtain the following corollary of the Theorem \ref{homcomp}, because in this case the condition of Theorem \ref{homcomp} is the requirement for $d\rho_{\mathcal{S}}(X)$ to be an automorphism of the Lie algebra $\fg$.

\begin{cor}\label{modhom}
The $\mathcal{S}$--automorphism group of the Cartan geometry $(G\to G/H,\om_G)$ of type $(\fg,H,\Ad_G|_H)$ is the group $$G\rtimes (\mathcal{S}\cap\{(d\sigma,\sigma|_H): \sigma\in \Aut(G),\ \sigma(H)\subset H\}).$$
\end{cor}

We can use similar argumentation to above for the general homogeneous Cartan geometries given by the extension $(\alpha,i)$ of $(\fg,H,\Ad_G|_H)$ to $(\fk,L,\rho)$ only in the case, when $X\in \fl\oplus \fs$ satisfies $d\rho_{\mathcal{S}}(X)(\alpha(\fg))\subset \alpha(\fg)$. However, in such case we obtain the following global result as Corollary of Theorem \ref{homcomp}.

\begin{cor}\label{autcor}
If $l\in L$ and $(\beta,j)\in Ext(\fk,L,\rho)$ are such that $\rho(l)\circ\beta$ preserves the image of $\alpha$. Then if the Lie algebra automorphism $\alpha^{-1}\circ \rho(l)\circ\beta\circ \alpha$ of $\fg$ induces a Lie group automorphism of $G$, then there is $(\beta,j)$--automorphism of $\mathcal{F}_{(\alpha,i)}(G\to G/H,\om_G)$ mapping the point $\iota(e)\in \mathcal{F}_{(\alpha,i)}(G)$ on $r^{l}\circ \iota(e)$.
\end{cor}

\section{Examples and applications}\label{sec6}

In this section, we will use the same notation $\Ad$ for any adjoint representation that is induced by the restriction of the adjoint representation of $\mathbb{R}^n\rtimes Gl(n)$ to a subgroup of $Gl(n)$.

\subsection{Homotheties of Riemannian geometries as morphisms induced by extension functors}\label{6,1}

We mentioned in the Section \ref{sec1} that all maximally symmetric models of the Riemannian geometries share the same skeleton $(\mathbb{R}^n\oplus \frak{so}(n),O(n),\Ad)$. We need to set a condition that will represent the torsion--freeness in our setting in order the view the category of Riemannian geometries as subcategory of the category $\mathcal{C}^{(\mathbb{R}^n\oplus \frak{so}(n),O(n),\Ad)}$. We can use the fact that $\mathbb{R}^n$ is the unique $O(n)$--invariant complement of $\frak{so}(n)$ in $\mathbb{R}^n\oplus \frak{so}(n)$.

\begin{def*}
We say that $T:=d\om|_{\wedge^2 (\mathbb{R}^n)^*\otimes \mathbb{R}^n}$ is a \emph{torsion of the Cartan geometry $(\ga\to M,\om)$ of type $(\mathbb{R}^n\oplus \frak{so}(n),O(n),\Ad)$}.
\end{def*}

We can define torsion in the analogous way for all Cartan geometries of type $(\fk,L,\rho)$ such that the representation $\rho$ is completely reducible and there is unique $L$--invariant complement of $\fl$ in $\fk$.

\begin{lem}\label{riecat}
There is equivalence of categories between the category of Riemannian geometries and the subcategory of $\mathcal{C}^{(\mathbb{R}^n\oplus \frak{so}(n),O(n),\Ad)}$ consisting of Cartan geometries satisfying $T=0$.
\end{lem}
\begin{proof}
In order to fix the notation for further purposes, we present here proof based on existence and properties of Levi--Civita connection, c.f. \cite{KoNo} for the explicit details.

Let $(\pi: \ga\to M,\om)$ be a Cartan geometry of type $(\mathbb{R}^n\oplus \frak{so}(n),O(n),\Ad)$. The natural inclusion of $\mathbb{R}^n\rtimes O(n)\subset \mathbb{R}^n\rtimes Gl(n)$ defines an extension functor from $\mathcal{C}^{(\mathbb{R}^n\oplus \frak{so}(n),O(n),\Ad)}$ to $\mathcal{C}^{(\mathbb{R}^n\oplus \frak{gl}(n),Gl(n),\Ad)}$, which we denote as $\mathcal{F}_{aff}$. Then we can decompose $\mathcal{F}_{aff}(\om)=\theta+\gamma$ according to the values to $\mathbb{R}^n$ and $\frak{gl}(n)$. The map $T_u\pi\circ \theta^{-1}: \mathbb{R}^n\to T_{\pi(u)}M$ identifies each $u\in \mathcal{F}_{aff}(\ga)$ with a frame of $T_{p(u)}M$ and thus we can identify $\mathcal{F}_{aff}(\ga)$ with the first order frame bundle $P^1M$. Then $\ga\subset P^1M$ is a $O(n)$--subbundle of $P^1M$, i.e., a Riemannian geometry. Further, $\gamma$ is the  connection form of the Levi--Civita connection of the Riemannian geometry if and only if $T=0$. Since $\mathcal{F}_{aff}$ is extension functor, the morphisms in $\mathcal{F}_{aff}(\mathcal{C}^{(\mathbb{R}^n\oplus \frak{so}(n),O(n),\Ad)})$ preserve the subbundle $\ga\subset P^1M$, i.e., they are morphisms of Riemannian geometries.

Conversely, the Levi--Civita connection $\gamma$ on $P^1M$ and the natural soldering form $\theta$ on $P^1M$ form together a Cartan connection $\theta+\gamma$ of type $(\mathbb{R}^n\oplus \frak{gl}(n),Gl(n),\Ad)$.  If $\ga\subset P^1M$ is a $O(n)$--subbundle, then the Levi--Civita connection $\gamma$ restricts to $\frak{so}(n)$--valued form on $T\ga$ and $(\theta+\gamma)|_{T\ga}$ is Cartan connection of type $(\mathbb{R}^n\oplus \frak{so}(n),O(n),\Ad)$ such that $T=0$. The morphisms of Riemannian geometries are morphisms of the Levi--Civita connection $\gamma$ that preserve the subbundle $\ga\subset P^1M$ and thus restrict to morphisms in $\mathcal{C}^{(\mathbb{R}^n\oplus \frak{so}(n),O(n),\Ad)}$.
\end{proof}

A homothety of Riemannian metrics $\ga\subset P^1M$ and $\ga'\subset P^1M'$ is diffeomorphism $f: M\to M'$ such that $(P^1f)^*(\ga')=r^{a\cdot \id}\circ \ga$ holds for some $a\cdot \id$ in the center $Z(Gl(n))$ of $Gl(n)$. Thus if $a\neq 1$, then $P^1f$ is not morphism in the category $\mathcal{C}^{(\mathbb{R}^n\oplus \frak{so}(n),O(n),\Ad)}$. On the other hand, we know that $r^{a\cdot \id}\in Ext(\mathbb{R}^n\oplus \frak{gl}(n),Gl(n),\Ad)\mathcal{C}^{(\mathbb{R}^n\oplus \frak{gl}(n),Gl(n),\Ad)}$ and the extension $(\Ad(a\cdot \id)^{-1},\id)$ restricts to extension in $Ext(\mathbb{R}^n\oplus \frak{so}(n),O(n),\Ad)$. Therefore we have proven the equivalence of categories (1) and (2) in the following Proposition.

\begin{prop}\label{riehom}
There is equivalence of categories between the following categories:

\begin{enumerate}
\item The category of Riemannian geometries with homotheties.
\item The subcategory of $Z(Gl(n))\mathcal{C}^{(\mathbb{R}^n\oplus \frak{so}(n),O(n),\Ad)}$ consisting of Cartan geometries satisfying $T=0$.
\item The subcategory of $\mathcal{C}^{(\mathbb{R}^n \oplus  \frak{co}(n),CO(n),\Ad)}$ consisting of Cartan geometries in the image of $\mathcal{F}_{\mathcal{S}}$ satisfying $T=0$ for $\mathcal{S}=Z(Gl(n))$.
\end{enumerate}
\end{prop}
\begin{proof}
Since for $\mathcal{S}=Z(Gl(n))$ is $(\mathbb{R}^n\oplus \frak{so}(n)\oplus \fs,O(n)\rtimes_{pr_2}\mathcal{S},(\Ad)_{\mathcal{S}})=(\mathbb{R}^n \oplus  \frak{co}(n),CO(n),\Ad)$ and $\mathcal{F}_{\mathcal{S}}$ is the extension functor induced by inclusion $\mathbb{R}^n\rtimes O(n)\subset \mathbb{R}^n \rtimes CO(n)$, the equivalence of categories (2) and (3) follows from the Theorem \ref{main1}.
\end{proof}

Let us remark that the image of $\mathcal{F}_{\mathcal{S}}$ does not exhaust the whole subcategory of $\mathcal{C}^{(\mathbb{R}^n \oplus  \frak{co}(n),CO(n),\Ad)}$ consisting of Cartan geometries satisfying $T=0$. Indeed, according to \cite[Section 1.6]{parabook}, Cartan geometry $(\ga\to M,\om)$ of type $(\mathbb{R}^n \oplus  \frak{co}(n),CO(n),\Ad)$ is in the image of $\mathcal{F}_{\mathcal{S}}$ if and only if there is a global $O(n)$--equivariant section $\ga/Z(Gl(n))\to \ga$. 

Since the inclusion $\mathbb{R}^n \rtimes CO(n)\subset \mathbb{R}^n\rtimes Gl(n)$ induces extension of $(\mathbb{R}^n \oplus  \frak{co}(n),CO(n),\Ad)$ to $(\mathbb{R}^n\oplus \frak{gl}(n),Gl(n),\Ad)$ assigns again to Cartan geometries in the image of $\mathcal{F}_{\mathcal{S}}$ the Levi--Civita connection of the original Riemannian metric, we get the following consequence.

\begin{cor}\label{homtoaff}
If $f$ is a homothety of Riemannian metrics, then $P^1f$ is an isomorphisms of the corresponding Levi--Civita connections.
\end{cor}

\subsection{Morphisms induced by extension functors and equivalence of Riemannian metrics with the same Levi--Civita connection}

Let us show, how can we use the $\mathcal{S}$--automorphism groups to compute the space of non--equivalent Riemannian metrics sharing the same Levi--Civita connection. In this section, we will use the notation introduced in the previous section \ref{6,1}.

Let $(P^1M\to M,\om=\theta+\gamma)$ be a Cartan geometry of type $(\mathbb{R}^n\oplus \frak{gl}(n),Gl(n),\Ad)$ satisfying $T=0$. There is the following description of the holonomy reductions (c.f. \cite{KoNo}) rephrased via Cartan geometries. 

Let $Hol(u)$ be the holonomy group of the principal connection $\gamma$ at $u\in P^1M$. Then the natural inclusion $$\iota: \mathbb{R}^n\rtimes Hol(u)\to \mathbb{R}^n\rtimes Gl(n)$$ induces extension $(d\iota,\iota|_{Hol(u)})$ of $(\mathbb{R}^n\oplus \frak{hol}(u),Hol(u),\Ad)$ to $(\mathbb{R}^n\oplus \frak{gl}(n),Gl(n),\Ad)$ and there is unique Cartan connection $\om^u$ of type $(\mathbb{R}^n\oplus \frak{hol}(u),Hol(u),\Ad)$ on the holonomy subbundle $\ga(u)\subset P^1M$ such that $$\iota_u^*(\om)=d\iota\circ \om^u$$ holds for the inclusion $\iota_u: \ga(u)\to P^1M$. In other words, $\om^u$ is the unique Cartan connection of type $(\mathbb{R}^n\oplus \frak{hol}(u),Hol(u),\Ad)$ on $\ga(u)$ such that $(P^1M\to M,\om)$ and $\mathcal{F}_{(d\iota,\iota|_{Hol(u)})}(\ga(u)\to M,\om^{u})$ are isomorphic Cartan geometries of type $(\mathbb{R}^n\oplus \frak{gl}(n),Gl(n),\Ad)$.

We show that we can relate automorphisms of the Cartan geometry $(P^1M\to M,\om)$ with certain $\mathcal{S}$--automorphisms of the holonomy reduction $(\ga(u)\to M,\om^{u})$.

\begin{lem}\label{5,3}
For $\Phi\in \Aut(P^1M,\om)$ and $u\in P^1M$, there is $p\in N_{Gl(n)}(Hol(u))$ such that $\Phi(\ga(u))=\ga(up)$ holds. In particular, $$r^{p^{-1}}\circ \Phi$$ is an $(\Ad(p),\conj(p))$--automorphism of $(\ga(u)\to M,\om^{u})$.
\end{lem}
\begin{proof}
Since the automorphism $\Phi$ of $(P^1M\to M,\om)$ preserves the horizontal distribution $\om^{-1}(\mathbb{R}^n)$, it holds $\Phi(\ga(u))=\ga(up)$ for some $p\in Gl(n)$ due to construction of the holonomy reduction. Then $p\in N_{Gl(n)}(Hol(u))$, because $$\ga(up)=\Phi(\ga(uh))=r^h\circ \Phi(\ga(u))=\ga(uph)$$ holds for all $h\in Hol(u)$. Thus the claim follows, because $(\Ad(p),\conj(p))\in Ext(\mathbb{R}^n\oplus \frak{hol}(u),Hol(u),\Ad)$ for $p\in N_{Gl(n)}(Hol(u))$ and thus $r^{p^{-1}}\circ \Phi$ is an $(\Ad(p),\conj(p))$--automorphisms of $(\ga(u)\to M,\om^{u})$.
\end{proof}

Let us denote $$\mathcal{S}^{Hol(u)}:=\{(\Ad(p),\conj(p)): p\in N_{Gl(n)}(Hol(u))\}.$$ Since the skeleton $(\mathbb{R}^n\oplus \frak{hol}(u),Hol(u),\Ad)$ is effective, we know from the Theorem \ref{effective} that $$N_ {\mathcal{S}^{Hol(u)}}=\{(h,(\Ad(h)^{-1},\conj(h)^{-1})): h\in Hol(u)\}$$ holds for the kernel of $(\mathbb{R}^n\oplus \frak{hol}(u)\oplus \fs^{Hol(u)},Hol(u)\rtimes_{pr_2}\mathcal{S}^{Hol(u)},(\Ad)_{\mathcal{S}^{Hol(u)}})$ and thus $(\mathbb{R}^n\oplus \fn_{Gl(n)}(Hol(u)),N_{Gl(n)}(Hol(u)),\Ad)$ is the effective quotient of $(\mathbb{R}^n\oplus \frak{hol}(u)\oplus \fs^{Hol(u)},Hol(u)\rtimes_{pr_2}\mathcal{S}^{Hol(u)},(\Ad)_{\mathcal{S}^{Hol(u)}})$, because we can identify element $(h,\Ad(p),\conj(p))\in Hol(u)\rtimes_{pr_2}\mathcal{S}^{Hol(u)}/N_ {\mathcal{S}^{Hol(u)}}$ with element $hp\in N_{Gl(n)}(Hol(u))$. Therefore we obtain the following application of the computation of the $\mathcal{S}$--automorphism group.

\begin{thm}\label{autext}
Let $(P^1M\to M,\om)$ be a Cartan geometry of type $(\mathbb{R}^n\oplus \frak{gl}(n),Gl(n),\Ad)$ satisfying $T=0$ and $(\ga(u)\to M,\om^u)$ be the holonomy reduction of $(P^1M\to M,\om)$ at $u\in P^1M$. Then
\begin{align*}
\Aut(P^1M,\om)&\cong \mathcal{S}^{Hol(u)}\Aut(\ga(u),\om^u)/\{r^h: h\in Hol(u)\}\\
&\cong \Aut(\mathcal{F}_{eff}\circ \mathcal{F}_{\mathcal{S}^{Hol(u)}}(\ga(u)\to M,\om^u)).
\end{align*}
In particular, $dim(\Aut(P^1M,\om))$ is at most $n+dim(\frak{s}^{Hol(u)})$.
\end{thm}
\begin{proof}
The map from $\Aut(P^1M,\om)$ to $\mathcal{S}^{Hol(u)}\Aut(\ga(u),\om^u)$ defined in Lemma \ref{5,3} depends on the choice of point in $\ga(u)$, but on the other hand the reduction $(\ga(uh)\to M,\om^{uh})$ at $uh$ for $h\in Hol(u)$ differs from $(\ga(u)\to M,\om^u)$ by the trival $\mathcal{S}^{Hol(u)}$--automorphism $r^h$. Therefore $\Aut(P^1M,\om)\subset \mathcal{S}^{Hol(u)}\Aut(\ga(u),\om^u)/\{r^h: h\in Hol(u)\}$ holds.

Conversely, the natural inclusion $\iota$ of $\mathbb{R}^n\oplus N_{Gl(n)}(Hol(u))$ to $\mathbb{R}^n\oplus Gl(n)$ induces extension $(d\iota,\iota|_{N_{Gl(n)}(Hol(u))})$ of $(\mathbb{R}^n\oplus \fn_{Gl(n)}(Hol(u)),N_{Gl(n)}(Hol(u)),\Ad)$ to $(\mathbb{R}^n\oplus \frak{gl}(n),Gl(n),\Ad)$ is extension and it is clear that  $\mathcal{F}_{(d\iota,\iota|_{\mathcal{S}^{Hol(u)}})}\circ\mathcal{F}_{eff}\circ \mathcal{F}_{\mathcal{S}^{Hol(u)}}(\ga(u)\to M,\om^u))=(P^1M\to M,\om)$ holds. Therefore the rest of the claim follows from the Theorem \ref{effective}.
\end{proof}

It is well--known (c.f. \cite{KoNo}) that $\gamma$ is the connection form of a Levi--Civita connection of some Riemannian metric on $M$ if and only if the holonomy group $Hol(u)$ is subgroup of $O(n)$ for some $u\in P^1M$. In particular, if $Hol(u)\subset O(n)$, then the inclusion $$\iota: \mathbb{R}^n\rtimes Hol(u)\to \mathbb{R}^n\rtimes O(n)$$ induces extension $(d\iota,\iota|_{Hol(u)})$ of $(\mathbb{R}^n\oplus \frak{hol}(u),Hol(u),\Ad)$ to $(\mathbb{R}^n\oplus \frak{so}(n),O(n),\Ad)$ and $\mathcal{F}_{(d\iota,\iota|_{Hol(u)})}(\ga(u)\to M,\om^{u})$ is a Cartan geometry encoding a Riemannian geometry according to Lemma \ref{riecat}. All natural inclusions $\iota$ mentioned in this section only differer by the choices of subgroups off $Gl(n)$ and thus $\gamma$ is the connection form of the Levi--Civita connection of the Riemannian geometry corresponding to $\mathcal{F}_{(d\iota,\iota|_{Hol(u)})}(\ga(u)\to M,\om^{u})$,

Of course, since the holonomy reduction $(\ga(u)\to M,\om^{u})$ depends on the choice of $u$, there can be many $u\in P^1M$ such that the holonomy group $Hol(u)$ is subgroup of $O(n)$. The problem is to determine, for which $u\in P^1M$ such that $Hol(u)\subset O(n)$ the Cartan geometries $\mathcal{F}_{(d\iota,\iota|_{Hol(u)})}(\ga(u)\to M,\om^{u})$ of type $(\mathbb{R}^n\oplus \frak{so}(n),O(n),\Ad)$ correspond to not equivalent Reimannian geometries.

Let us fix point $u\in P^1M$ and assume that $Hol(u)\subset O(n)$ holds for the holonomy reduction $(\ga(u)\to M,\om^{u})$ of $(P^1M\to M,\om)$ in the rest of this section. If $u'\in \ga(u)$, then $(\ga(u')\to M,\om^{u'})=(\ga(u)\to M,\om^{u})$, and if $p\in Gl(\mathbb{R}^n)$, then $(\ga(up)\to M,\om^{up})$ is Cartan geometry of type $(\mathbb{R}^n\rtimes \conj(p)^{-1}(Hol(u)),\conj(p)^{-1}(Hol(u)))$. Therefore we need to determine for which $p\in Gl(n)$ satisfying $\conj(p)^{-1}(Hol(u))\subset O(n)$ are the Cartan geometries $\mathcal{F}_{(d\iota,\iota|_{Hol(u)})}(\ga(u)\to M,\om^{u})$ and $\mathcal{F}_{(d\iota,\iota|_{\conj(p)^{-1}Hol(u)})}(\ga(up)\to M,\om^{up})$ of type $(\mathbb{R}^n\oplus \frak{so}(n),O(n),\Ad)$ isomorphic. Clearly, if $uk\in \ga(up)$ for some $k\in O(n)$, the Cartan geometries of type $(\mathbb{R}^n\oplus \frak{so}(n),O(n),\Ad)$ provide the same reduction of $P^1M$ to $O(n)$, i.e., the same Riemannian metric, which leads to the following well--known result.

\begin{lem}\label{reddif}
The set of Riemannian metrics on $M$ such that the connection forms of their Levi--Civita connections coincide with $\gamma$ is in bijective correspondence with the set $$Z_{\exp(S^2\mathbb{R}^n)}(Hol(u)).$$
\end{lem}
\begin{proof}
If the Cartan geometries $\mathcal{F}_{(d\iota,\iota|_{Hol(u)})}(\ga(u)\to M,\om^{u})$ and $\mathcal{F}_{(d\iota,\iota|_{\conj(p)^{-1}Hol(u)})}(\ga(up)\to M,\om^{up})$ of type $(\mathbb{R}^n\oplus \frak{so}(n),O(n),\Ad)$ correspond to different Riemannian metrics, then there is no $k\in O(n)$ such that $uk\in \ga(up)$. Therefore it suffices to prove that the set $Z_{\exp(S^2\mathbb{R}^n)}(Hol(u))$ is in bijective correspondence with the set of elements $p\in Gl(n)/O(n)$ such that $\conj(p)^{-1}(Hol(u))\subset O(n)$.

We recall that there are the Cartan involutions $d\theta(A)=-A^T$ of $\frak{gl}(n)$ and $\theta(A)=(A^T)^{-1}$ of $Gl(n)$ with fixed point sets $\frak{so}(n)$ and $O(n)$, respectively. Since $S^2 \mathbb{R}^n$ is the $-1$--eigenspace of $d\theta$ and $\exp$ is a bijection between $S^2 \mathbb{R}^n$ and the space of all non--degenerate positive--definite symmetric matrices, the global Cartan decomposition of $Gl(n)$ is of the form $Gl(n)=\exp(S^2 \mathbb{R}^n)O(n)$. Let us remark that this decomposition is also known as polar decomposition. Therefore will decompose $p\in Gl(n)$ as $\exp(X)k$ for $\exp(X)\in \exp(S^2 \mathbb{R}^n)$ and $k\in O(n)$.

Therefore it suffices to prove that the set $Z_{\exp(S^2\mathbb{R}^n)}(Hol(u))$ is in bijective correspondence with the set of elements $\exp(X)\in \exp(S^2\mathbb{R}^n)$ such that $$\conj(\exp(X))^{-1}(Hol(u))\subset O(n).$$

If $\conj(\exp(X))^{-1}(Hol(u))\subset O(n)$ for $X\in S^2 \mathbb{R}^n$, then $\exp(-X) h \exp(X)=\theta(\exp(-X) h \exp(X))=\exp(X) h \exp(-X)$ holds for all $h\in Hol(u)$. Thus $\exp(2X)$ centralizes $Hol(u)$ and $\exp(X)\in Z_{\exp(S^2\mathbb{R}^n)}(Hol(u))$, because $\exp(X)$ is diagonalizable.
\end{proof}

So it remains to show, when two elements of $Z_{\exp(S^2\mathbb{R}^n)}(Hol(u))$ correspond to equivalent Riemannian metrics. From the Corollary \ref{homtoaff} follows that an isomorphism of two Riemannian metrics on $M$ such that the connection forms of their Levi--Civita connections coincide with $\gamma$ is an automorphism of $(P^1M\to M,\om)$. Therefore we can apply the Theorem \ref{autext} to obtain the second application of the computation of $\mathcal{S}$--automorphism group.

\begin{thm}\label{meteq}
Let $(\ga(u)\to M,\om^{u})$ be a holonomy reduction of $(P^1M\to M,\om)$ such that $Hol(u)\subset O(n)$ for some $u\in P^1M$. Then two Riemannian metrics corresponding to $\exp(Y_1),\exp(Y_2)\in Z_{\exp(S^2\mathbb{R}^n)}(Hol(u))$ are isomorphic if and only if there is an $(\Ad(\exp(Y_1)k\exp(-Y_2)),\conj(\exp(Y_1)k\exp(-Y_2)))$--automorphism of $(\ga(u)\to M,\om^{u})$ for some $k\in N_{O(n)}(Hol(u))$.

Moreover, for each $\exp(Y)\in Z_{\exp(S^2\mathbb{R}^n)}(Hol(u))$ and each $p\in N_{Gl(n)}(Hol(u))$ there is unique $\exp(Y_p)\in Z_{\exp(S^2\mathbb{R}^n)}(Hol(u))$ such that $p=\exp(Y)k\exp(-Y_p)$ holds for some $k\in N_{O(n)}(Hol(u))$. In particular, the assignment $$\tau(\Phi)(\exp(Y)):=\exp(Y_p)$$ for $(\Ad(p),\conj(p))$--automorphism $\Phi$ of $(\ga(u)\to M,\om^{u})$ is a right action of the Lie group $\mathcal{S}^{Hol(u)}\Aut(\ga(u),\om^u)$ on $Z_{\exp(S^2\mathbb{R}^n)}(Hol(u))$ and the space of non--equivalent Riemanian metrics on $M$ such that the connection forms of their Levi--Civita connections coincide with $\gamma$ is in bijection with the orbit space $$Z_{\exp(S^2\mathbb{R}^n)}(Hol(u))/\tau(\mathcal{S}^{Hol(u)}\Aut(\ga(u),\om^u)).$$
\end{thm}
\begin{proof}
It is clear that $\exp(Y_1)k\exp(-Y_2)\in N_{Gl(n)}(Hol(u))$ holds for $\exp(Y_1),\exp(Y_2)\in Z_{\exp(S^2\mathbb{R}^n)}(Hol(u))$ and $k\in N_{O(n)}(Hol(u))$ and thus $$(\Ad(\exp(Y_1)k\exp(-Y_2)),\conj(\exp(Y_1)k\exp(-Y_2)))\in \mathcal{S}^{Hol(u)}.$$ We know from the Theorem \ref{autext} that the $(\Ad(\exp(Y_1)k\exp(-Y_2)),\conj(k))$--automorphisms of $(\ga(u)\to M,\om^{u})$ are in bijective correspondence with automorphisms of $(P^1M\to M,\om)$ mapping $\ga(u)$ on $\ga(u\exp(Y_2)k^{-1}\exp(-Y_1))$.

Thus if there is an $(\Ad(\exp(Y_1)k\exp(-Y_2)),\conj(k))$--automorphism, then the corresponding automorphisms of $(P^1M\to M,\om)$ maps $\ga(u\exp(Y_1))$ on $\ga(u\exp(Y_2)k^{-1})$ and the Riemannian metrics corresponding to $\exp(Y_1),\exp(Y_2)$ are isomorphic.

Conversely, if the Riemannian metrics corresponding to $\exp(Y_1),\exp(Y_2)$ are isomorphic, then the corresponding automorphisms of $(P^1M\to M,\om)$ maps $\ga(u\exp(Y_1))$ on $\ga(u\exp(Y_2)k^{-1})$ for some $k\in O(n)$ such that $\exp(Y_1)k\exp(-Y_2)\in N_{Gl(n)}(Hol(u))$ and $k\in N_{O(n)}(Hol(u))$, because $\conj(\exp(Y_1)k\exp(-Y_2))(Hol(u))=Hol(u)$ implies $\conj(k)(Hol(u))=Hol(u)$.

In general, if $p=\exp(Y_2)k\in N_{Gl(n)}(Hol(u))$, then $\exp(Y_2)\in Z_{\exp(S^2\mathbb{R}^n)}(Hol(u))$ follows as in the proof of Lemma \ref{reddif} and thus $k\in N_{O(n)}(Hol(u))$. Therefore $\exp(Y_p)$ is uniquely given by the analogous decomposition $p^{-1}\exp(Y)=\exp(Y_p)k^{-1}$. Moreover, $$\exp((Y_{p_1})_{p_2})=p_2^{-1}\exp(Y_{p_1})k_2=p_2^{-1}p_1^{-1}\exp(Y)k_1k_2=\exp(Y_{p_1p_2})$$
holds and thus the map $\tau$ is composition of the group homomorphism $$\mathcal{S}^{Hol(u)}\Aut(\ga(u),\om^u)\to \mathcal{S}^{Hol(u)}$$ assigning to $(\alpha,i)$--automorphism the extension $(\alpha,i)$ and the right action $Y\mapsto Y_p$.
\end{proof}

Let us point out that the trivial automorphisms in $\mathcal{S}^{Hol(u)}\Aut(\ga(u),\om^u)$ are in the kernel of $\tau$ and therefore it is enough to compute the group $\Aut(\mathcal{F}_{eff}\circ \mathcal{F}_{\mathcal{S}^{Hol(u)}}(\ga(u)\to M,\om^u))$ for the Cartan geometry $(\ga(u)\to M,\om^{u})$ of type $(\mathbb{R}^n\oplus \frak{hol}(u),Hol(u),\Ad)$. We obtain the following claim as a corollary of the construction of the skeleton $(\mathbb{R}^n\oplus \fn_{Gl(n)}(Hol(u)),N_{Gl(n)}(Hol(u)),\Ad)$ and the Theorem \ref{autcomp}.

\begin{cor}\label{explcomp}
\begin{enumerate}
\item The extension corresponding to the extension functor $\mathcal{F}_{eff}\circ \mathcal{F}_{\mathcal{S}^{Hol(u)}}$ is induced by the inclusion of $\mathbb{R}^n\oplus \frak{hol}(u)$ to $\mathbb{R}^n\oplus \frak{s}^{Hol(u)}$ and the linear connection $\nabla^A$ from the Theorem \ref{autcomp} is the adjoint tractor connection for the Cartan geometry $\mathcal{F}_{eff}\circ \mathcal{F}_{\mathcal{S}^{Hol(u)}}(\ga(u)\to M,\om^{u})$ of type $(\mathbb{R}^n\oplus \fn_{Gl(n)}(Hol(u)),N_{Gl(n)}(Hol(u)),\Ad)$.
\item Infinitesimal automorphisms of $\mathcal{F}_{eff}\circ \mathcal{F}_{\mathcal{S}^{Hol(u)}}(\ga(u)\to M,\om^{u})$ are sections $s$ of $\ga(u)\times_{Hol(u)}(\mathbb{R}^n\oplus \fn_{Gl(n)}(Hol(u)))$ such that
$$(\nabla^A s)_{\mathbb{R}^n}=0$$
and
$$(\nabla^A s)_{\fn_{Gl(n)}(Hol(u))}=-ins_{s_{\mathbb{R}^n}}\circ \kappa$$
hold.
\item If $s$ is a complete infinitesimal automorphism of $\mathcal{F}_{eff}\circ \mathcal{F}_{\mathcal{S}^{Hol(u)}}(\ga(u)\to M,\om^{u})$, then the metrics $\exp(Y)\in Z_{\exp(S^2\mathbb{R}^n)}(Hol(u))$ and $\tau(\exp(w.s_{\fn_{Gl(n)}(Hol(u))})(\exp(Y))$ are equivalent for all $w\in \mathbb{R}$.
\end{enumerate}
\end{cor}

Let us compute one non--trivial example.

\begin{exam}
Let us choose $G=SO(3)\times \mathbb{R}^+$ and $H=SO(2)\times \{1\}$. Then $(\alpha,i)$ given by formulas
$$i( \begin{pmatrix} cos(t) & -sin(t)\cr sin(t) & cos(t)\end{pmatrix}\times \{1\})=\begin{pmatrix}1 & 0 & 0 & 0\cr 0  & cos(t) & -sin(t) & 0\cr 0 & sin(t) & cos(t) & 0\cr 0 & 0 & 0 & 1\end{pmatrix}$$
$$\alpha(\begin{pmatrix}0 & -a & -b\cr a & 0 & -c\cr b & c & 0\end{pmatrix}\oplus \{x\})=\begin{pmatrix}0 & 0 & 0 & 0\cr a & 0 & -c & 0\cr b & c & 0 & 0\cr x & 0 & 0 & 0\end{pmatrix}$$
is an extension of $(\frak{so}(3)\oplus \mathbb{R},SO(2)\times\{1\},\Ad)$ to $(\mathbb{R}^3\oplus \frak{gl}(3),Gl(3),\Ad)$. 

Clearly, $Hol(e,1)=SO(2)\subset O(3)$ holds for the holonomy group of the homogeneous Cartan geometry $\mathcal{F}_{(\alpha,i)}(G\to G/H,\om_G)$ of type $(\mathbb{R}^3\oplus \frak{gl}(3),Gl(3),\Ad)$ at the point $\{e\}\times \{1\}$. Moreover, the extension $(\alpha,i)$ restricts to an extension of $(\frak{so}(3)\oplus \mathbb{R},SO(2)\times\{1\},\Ad)$ to $(\mathbb{R}^3\oplus \frak{hol}(e,1),Hol(e,1),\Ad)$ and provides the holonomy reduction.

The normalizer $N_{Gl(3)}(Hol(e,1))$ can be computed as $$N_{Gl(3)}(Hol(e,1))=\begin{pmatrix}e^{m_1}cos(t) & -e^{m_1}sin(t) & 0\cr e^{m_1}sin(t) & e^{m_1}cos(t) & 0\cr 0 & 0 & e^{m_2}\end{pmatrix}$$ and the set of Riemannian metrics on $G/H$ sharing linear connection given by extension $(\alpha,i)$ is according to the Lemma \ref{reddif} the set
$$Z_{\exp(S^2\mathbb{R}^3)}(Hol(e,1))=\begin{pmatrix} e^{m_1} & 0 & 0\cr 0 & e^{m_1} & 0\cr  0 & 0 & e^{m_2}\end{pmatrix}.$$

We compute the group $\Aut(\mathcal{F}_{eff}\circ \mathcal{F}_{\mathcal{S}(Hol(e,1))}(G\to G/H,\om_G))$ and determine, which metrics in $Z_{\exp(S^2\mathbb{R}^3}(Hol(e,1))$ are equivalent. Since we are on a homogeneous Cartan geometry, we can use the Theorem \ref{homcomp} for the following map $A: \fg\to Gl(\mathbb{R}^3\oplus \fn_{Gl(3)}(Hol(e,1)))$
\begin{align*}
A(\begin{pmatrix}0 & -a & -b\cr a & 0 & -c\cr b & c & 0\end{pmatrix}\oplus \{x\})(\begin{pmatrix}0 & 0 & 0 & 0\cr a' & m_1' & -c' & 0\cr b' & c' & m_1' & 0\cr x' & 0 & 0 & m_2'\end{pmatrix}):=\\
\begin{pmatrix}0 & 0 & 0 & 0\cr -a\,m_1'+b\,c'-c\,b' & 0 & b\,a'-a\,b' & 0\cr -b\,m_1'-a\,c'+c\,a' & a\,b'-b\,a' & 0 & 0\cr -x\,m_2' & 0 & 0 & 0\end{pmatrix}
\end{align*}
and we can solve that the local infinitesimal automorphisms of $\Aut(\mathcal{F}_{eff}\circ \mathcal{F}_{\mathcal{S}(Hol(e,1))}(G\to G/H,\om_G))$ are given by the set $$\begin{pmatrix}0 & 0 & 0 & 0\cr a' & 0 & -c' & 0\cr b' & c' & 0 & 0\cr x' & 0 & 0 & m_2'\end{pmatrix} \subset \mathbb{R}^3\oplus \fn_{Gl(3)}(Hol(e,1)).$$
However, since the elements of $N_{Gl(3)}Hol(e,1)$ preserve the image of $\alpha$, we can use directly the Corollary \ref{autcor} to compute that $\Aut(\mathcal{F}_{eff}\circ \mathcal{F}_{\mathcal{S}(Hol(e,e))}(G\to G/H,\om_G))=SO(3)\times (\mathbb{R}\rtimes Gl^+(1))$.

Thus we can conclude from Theorem \ref{meteq} that the set of non--equivalent Riemannian metrics on $G/H$ sharing the linear connection given by the extension $(\alpha,i)$ is the set
$$\begin{pmatrix} e^{m_1} & 0 & 0\cr 0 & e^{m_1} & 0\cr  0 & 0 & 1\end{pmatrix}.$$
\end{exam}

\end{document}